\documentclass[11pt,a4paper]{article}
\usepackage{amsmath,amssymb,amsthm,amscd,mathrsfs}
\usepackage{indentfirst}
\usepackage[top=25mm, bottom=30mm, left=30mm, right=30mm]{geometry}
\newtheorem{Def}{\bf Definition}[subsection]
\newtheorem{Thm}[Def]{\bf Theorem}
\newtheorem{Lem}[Def]{\bf Lemma}
\newtheorem{Cor}[Def]{\bf Corollary}
\newtheorem{Pro}[Def]{\bf Proposition}
\newtheorem{Rem}[Def]{\bf Remark}

\newtheorem{ThmA}{\bf Theorem}
\newtheorem{CorA}{\bf Corollary}

\title{\bf On bi-exactness of discrete quantum groups}
\author{Yusuke Isono\thanks{Department of Mathematical Sciences,
University of Tokyo, Komaba, Tokyo, 153-8914, Japan \protect \\  E-mail: \texttt{isono@ms.u-tokyo.ac.jp}}}
\date{}
\begin{document}
\maketitle

\begin{abstract}
We define Ozawa's notion of bi-exactness to discrete quantum groups, and then prove some structural properties of associated von Neumann algebras. In particular, we prove that any non amenable subfactor of free quantum group von Neumann algebras, which is an image of a faithful normal conditional expectation, has no Cartan subalgebras.
\end{abstract}

\section{\bf Introduction}


A countable discrete group $\Gamma$ is said to be \textit{bi-exact} (or said to be in \textit{class} $\cal S$) 
if it is exact and there exists a map $\mu\colon \Gamma \rightarrow \mathrm{Prob}(\Gamma)\subset \ell^1(\Gamma)$ such that  $\limsup_{x\rightarrow\infty}\|\mu(sxt)-s\mu(x)\|_1=0$ for any $s,t\in \Gamma$. 
This notion was introduced and studied by Ozawa in his book with Brown  $\cite[\textrm{Section 15}]{BO08}$. In particular, he gave the following two different characterizations of bi-exactness (Lemma 15.1.4 and Proposition 15.2.3(2) in the book):
\begin{itemize}
	\item[$\rm (i)$] The group $\Gamma$ is exact and the algebra $L\Gamma$ satisfies condition $\rm (AO)^+$ with the dense $C^*$-algebra $C^*_\lambda(\Gamma)$.		
	\item[$\rm (ii)$] There exists a unital $C^*$-subalgebra ${\cal B}\subset \ell^\infty (\Gamma)$ such that  
	\begin{itemize}
		\item[$\bullet$] the algebra $\cal B$ contains $c_0(\Gamma)$ (so that we can define ${\cal B}_\infty:={\cal B}/c_0(\Gamma)$);
		\item[$\bullet$] the left translation action on $\ell^\infty(\Gamma)$ induces an amenable action on ${\cal B}_{\infty}$, and the right one induces the trivial action on ${\cal B}_{\infty}$.
	\end{itemize}
\end{itemize}
Here we recall that a von Neumann algebra $M\subset \mathbb{B}(H)$ with a standard representation satisfies \textit{condition} $\rm (AO)^+$ $\cite[\textrm{Definition 3.1.1}]{Is12_2}$ if there exists a unital $\sigma$-weakly dense locally reflexive $C^*$-subalgebra $A\subset M$ and a unital completely positive (say, u.c.p.) map $\theta\colon A\otimes A^{\rm op} \rightarrow \mathbb{B}(H)$ such that $\theta(a\otimes b^{\rm op})-ab^{\rm op}\in \mathbb{K}(H)$ for any $a,b \in A$. 
Here $A^{\rm op}$ means the opposite algebra of $A$, which acts canonically on $H$ by right.

These characterizations of bi-exactness have been used in different contexts. For example, condition (i) is very close to condition (AO) introduced in $\cite{Oz03}$. In particular Ozawa's celebrated theorem in the same paper says that von Neumann algebras of bi-exact groups are solid, 
meaning that any relative commutant of a diffuse amenable subalgebra is still amenable.
Condition (ii) is used to show that hyperbolic groups are bi-exact. This follows from the fact that all hyperbolic groups act on their Gromov boundaries as amenable actions. 
The definition of bi-exactness itself is also interesting since it is closely related to an \textit{array}, which was introduced and studied in $\cite{CS11}$ and $\cite{CSU11}$.

In the present paper, we generalize bi-exactness to discrete quantum groups in two different ways, which correspond to conditions (i) and (ii). 
This is not a difficult task since condition (i) does not rely on the structure of group $C^*$-algebras (this is a $C^*$-algebraic condition) and all objects in  condition (ii) are easily defined for discrete quantum groups. 
We then study some basic facts on these conditions. In particular, we prove that free products of free quantum groups or quantum automorphism groups are bi-exact, showing that a condition close to bi-exactness is closed under taking free products. 

After these observations, we prove some structural properties of von Neumann algebras associated with bi-exact quantum groups. 
We first give the following theorem which generalizes $\cite[\textrm{Theorem 3.1}]{PV12}$. 
This is a natural generalization from discrete groups to discrete quantum groups of Kac type. 
In the proof, we need only slight modifications since the original proof for groups do not rely on the structures of group von Neumann algebras very much.
Note that a special case of this theorem was already generalized in $\cite{Is12_2}$ for general $\rm II_1$ factors.  
\begin{ThmA}\label{A}
Let $\mathbb{G}$ be a compact quantum group of Kac type, whose dual acts on a tracial von Neumann algebra $B$ as a trace preserving action. Write $M:=\hat{\mathbb{G}}\ltimes B$. 
Let $q$ be a projection in $M$ and $A\subset qMq$ a von Neumann subalgebra, which is amenable relative to $B$ in $M$. Assume that $\hat{\mathbb{G}}$ is weakly amenable and bi-exact. Then we have either one of the following statements:
\begin{itemize}
	\item[$\rm (i)$] We have $A\preceq_{M} B$.
	\item[$\rm (ii)$] The algebra ${\cal N}_{qMq}(A)''$ is amenable relative to $B$ in $M$ (or equivalently, $L^2(qM)$ is left ${\cal N}_{qMq}(A)''$-amenable as a $qMq$-$B$-bimodule.).
\end{itemize}
\end{ThmA}

The same arguments as in the group case give the following corollary. As we mentioned, free quantum groups and quantum automorphism groups are examples of the corollary (see also Theorem \ref{C}). 
\begin{CorA}
Let $\mathbb{G}$ be a compact quantum group of Kac type. Assume that $\hat{\mathbb{G}}$ is weakly amenable and bi-exact. 
\begin{itemize}
	\item[$\rm (1)$] The algebra $L^\infty(\mathbb{G})$ is strongly solid. Moreover any non amenable von Neumann subalgebra of $L^\infty(\mathbb{G})$ has no Cartan subalgebras.
	\item[$\rm (2)$] If $\hat{\mathbb{G}}$ is non amenable, then $L^\infty(\mathbb{G})\otimes B$ has no Cartan subalgebras for any finite von Neumann algebra $B$.
\end{itemize}
\end{CorA}

We also mention that if a non-amenable, weakly amenable and bi-exact discrete quantum group $\hat{\mathbb{G}}$ admits an action on a commutative von Neumann algebra $L^\infty(X,\mu)$, 
so that $\hat{\mathbb{G}}\ltimes L^\infty(X,\mu)$ is a $\rm II_1$ factor and $L^\infty(X,\mu)$ is a Cartan subalgebra, then $L^\infty(X,\mu)$ is the unique Cartan subalgebra up to unitary conjugacy.
However such an example is not known.

Next we investigate similar results on discrete quantum groups of non Kac type. For this, let us recall condition $\rm (AOC)^+$, which is a similar one to condition $\rm (AO)^+$ on continuous cores. 
We say a von Neumann algebra $M$ and a faithful normal state $\phi$ on $M$ satisfies \textit{condition} $\rm (AOC)^+$ $\cite[\textrm{Definition 3.2.1}]{Is12_2}$ if there exists a unital $\sigma$-weakly dense $C^*$-subalgebra $A\subset M$ such that 
the modular action of $\phi$ gives a norm continuous action on $A$ (so that we can define $A \rtimes_{\rm r} \mathbb{R}$), $A \rtimes_{\rm r} \mathbb{R}$ is locally reflexive, and there exists a u.c.p.\ map 
$\theta\colon (A\rtimes_{\rm r}\mathbb{R})\otimes (A\rtimes_{\rm r}\mathbb{R})^{\rm op} \rightarrow \mathbb{B}(L^2(M,\phi)\otimes L^2(\mathbb{R}))$
such that  $\theta(a\otimes b^{\rm op})-ab^{\rm op}\in \mathbb{K}(L^2(M)\otimes\mathbb{B}(L^2(\mathbb{R}))$ for any $a,b\in A\rtimes_{\rm r}\mathbb{R}$.

In $\cite{Is12_2}$, we proved that von Neumann algebras of free quantum groups with Haar states satisfy condition $\rm (AOC)^+$ and then deduced that they do not have modular action (of Haar states) invariant Cartan subalgebras if they have the $\rm W^*$CBAP (and they are non-amenable). 
This is a partial answer for the absence of Cartan subalgebras in two senses: the $\rm W^*$CBAP of free quantum groups were not known, and we could prove only the absence of special Cartan subalgebras under the $\rm W^*$CBAP.

Very recently De Commer, Freslon and Yamashita solved the first problem. In fact, they proved that free quantum groups and quantum automorphism groups are weakly amenable and hence von Neumann algebras of them have the $\rm W^*$CBAP $\cite{DFY13}$. 
Thus only the second problem remains to be proved.

In this paper, we solve the second problem. In fact, the following theorem is an analogue of Theorem \ref{A} on continuous cores and it gives a complete answer for our Cartan problem (except for amenability). 
In the theorem below, $C_h(L^\infty(\mathbb{G}))$ means the continuous core of $L^\infty(\mathbb{G})$ with respect to the Haar state $h$.
\begin{ThmA}\label{B}
Let $\mathbb{G}$ be a compact quantum group, $h$ a Haar state of $\mathbb{G}$ and $(B,\tau_B)$ a tracial von Neumann algebra. Denote $M:= C_h(L^\infty(\mathbb{G}))\otimes B$ and $\mathrm{Tr}_M:=\mathrm{Tr} \otimes\tau_B$, where $\mathrm{Tr}$ is the canonical trace on $C_h(L^\infty(\mathbb{G}))$.
Let $q$ be a $\mathrm{Tr}_M$-finite projection in $M$ and $A\subset qMq$ an amenable von Neumann subalgebra. 
Assume that $L^\infty(\mathbb{G})$ has the $ W^*$CBAP and $(L^\infty(\mathbb{G}),h)$  satisfies condition $\rm (AOC)^+$ with the dense $C^*$-algebra $C_{\rm red}(\mathbb{G})$. Then we have either one of the following statements:
\begin{itemize}
	\item[$\rm (i)$] We have $A\preceq_{M} L\mathbb{R}\otimes B$.
	\item[$\rm (ii)$] The algebra $L^2(qM)$ is left ${\cal N}_{qMq}(A)''$-amenable as a $qMq$-$B$-bimodule.
\end{itemize}
\end{ThmA}

\begin{CorA}
Let $\mathbb{G}$ be a compact quantum group. 
Assume that $L^\infty(\mathbb{G})$ has the $ W^*$CBAP and $(L^\infty(\mathbb{G}),h)$ satisfies condition $\rm (AOC)^+$ with the dense $C^*$-subalgebra $C_{\rm red}(\mathbb{G})$.

\begin{itemize}
	\item[$\rm (1)$] 
Any non amenable von Neumann subalgebra of $L^\infty(\mathbb{G})$, which is an image of a faithful normal conditional expectation, has no Cartan subalgebras.
	\item[$\rm (2)$] If $L^\infty(\mathbb{G})$ is non amenable, then $L^\infty(\mathbb{G})\otimes B$ has no Cartan subalgebras for any finite von Neumann algebra $B$.
\end{itemize}
\end{CorA}

As further examples of corollaries above, we give the following theorem. Note that this theorem does not say anything about amenability of $\hat{\mathbb{G}}$ and $L^\infty(\mathbb{G})$.

\begin{ThmA}\label{C}
Let $\mathbb{G}$ be one of the following compact quantum groups.
\begin{itemize}
	\item[$\rm (i)$] A co-amenable compact quantum group.
	\item[\rm (ii)] The free unitary quantum group $A_u(F)$ for any $F\in\mathrm{GL}(n,\mathbb{C})$.
	\item[\rm (iii)] The free orthogonal quantum group $A_o(F)$ for any $F\in\mathrm{GL}(n,\mathbb{C})$.
	\item[\rm (iv)] The quantum automorphism group $A_{\rm aut}(B, \phi)$ for any finite dimensional $C^*$-algebra $B$ and any faithful state $\phi$ on $B$.
	\item[\rm (v)] The dual of a bi-exact and weakly amenable discrete group $\Gamma$ with $\Lambda_{\rm cb}(\Gamma)=1$.
	\item[\rm (vi)] The dual of a free product $\hat{\mathbb{G}}_1*\cdots *\hat{\mathbb{G}}_n$, where each $\mathbb{G}_i$ is as in from $\rm (i)$ to $\rm (v)$ above.
\end{itemize}
Then the dual $\hat{\mathbb{G}}$ is weakly amenable and bi-exact.
The associated von Neumann algebra $L^\infty(\mathbb{G})$ has the $W^*$CBAP and satisfies condition $\rm (AOC)^+$ with the Haar state and the dense $C^*$-subalgebra $C_{\rm red}(\mathbb{G})$.
\end{ThmA}

\bigskip

\noindent
Throughout the paper, we always assume that discrete groups are countable, quantum group $C^*$-algebras are separable, von Neumann algebras have separable predual, and Hilbert spaces are separable.

\bigskip

\noindent 
{\bf Acknowledgement.} The author would like to thank Yuki Arano, Cyril Houdayer, Yasuyuki Kawahigashi, Narutaka Ozawa, Sven Raum and Stefaan Vaes for fruitful conversations. 
He was supported by JSPS, Research Fellow of the Japan Society for the Promotion of Science and FMSP, Frontiers of Mathematical Sciences and Physics. 
This research was carried out while he was visiting the Institut de Math$\rm \acute{e}$matiques de Jussieu. He gratefully acknowledges the kind hospitality of them.

\section{\bf Preliminaries}

\subsection{\bf Compact (and discrete) quantum groups}\label{CQG}

Let $\mathbb{G}$ be a compact quantum group. In the paper, we basically use the following notation.
We denote the comultiplication by $\Phi$, Haar state by $h$, the set of all equivalence classes of all unitary corepresentations by $\mathrm{Irred}(\mathbb{G})$, and right and left regular representations by $\rho$ and $\lambda$ respectively. 
We regard $C_{\rm red}(\mathbb{G}):=\rho(C(\mathbb{G}))$ as our main object and we frequently omit $\rho$ when we see the dense Hopf $*$-algebra. 
The canonical unitary $\mathbb{V}$ is defined as $\mathbb{V}=\bigoplus_{x\in\mathrm{Irred}\mathbb{G}} U^x$.
The GNS representation of $h$ is written as $L^2(\mathbb{G})$ and it has a decomposition $L^2(\mathbb{G})=\sum_{x\in\mathrm{Irred}(\mathbb{G})}\oplus (H_x\otimes H_{\bar{x}})$.
Along the decomposition, the modular operator of $h$ is of the form $\Delta^{it}=\sum_{x\in\mathrm{Irred}(\mathbb{G})}\oplus (F_x^{it}\otimes F_{\bar{x}}^{-it})$. 
The canonical conjugation is denoted by $J$. 
All dual objects are written with hat (e.g.\ $\hat{\mathbb{G}}, \hat{\Phi},\ldots$).
We frequently use the unitary element $U:=J\hat{J}$ which satisfies  $U\rho(C(\mathbb{G}))U=\lambda(C(\mathbb{G}))$ and $U\hat{\rho}(c_0(\hat{\mathbb{G}}))U=\hat{\lambda}(c_0(\hat{\mathbb{G}}))$.

\bigskip

{\bf $\bullet$ Crossed products}

\bigskip

For crossed products of quantum groups, we refer the reader to $\cite{BS93}$.

Let $\hat{\mathbb{G}}$ be a discrete quantum group and $A$ a $C^*$-algebra. Recall following notions:
\begin{itemize}
	\item A (\textit{left}) \textit{action} of $\hat{\mathbb{G}}$ on $A$ is a non-degenerate $*$-homomorphism $\alpha\colon A \rightarrow M(c_0(\hat{\mathbb{G}})\otimes A)$ such that  $(\iota\otimes \alpha)\alpha=(\hat{\Phi}\otimes\iota)\alpha$ and $(\hat{\epsilon}\otimes\iota)\alpha(a)=a$ for any $a\in A$, where $\hat{\epsilon}$ is the counit.
	\item A \textit{covariant representation} of an action $\alpha$ into a $C^*$-algebra $B$ is a pair $(\theta,X)$ such that 
	\begin{itemize}
 		\item $\theta$ is a non-degenerate $*$-homomorphism from $A$ into $M(B)$;
		\item $X \in M(c_0(\hat{\mathbb{G}})\otimes B)$ is a unitary representation of $\hat{\mathbb{G}}$;
		\item they satisfy the covariant relation $(\iota\otimes \theta)\alpha(a)=X^*(1\otimes \theta(a))X$ ($a\in A$).
	\end{itemize}
\end{itemize}

Let $\alpha$ be an action of $\hat{\mathbb{G}}$ on $A$. Then for a covariant representation $(\theta,X)$ into $B$, the closed linear span
\begin{equation*}
\overline{\mathrm{span}}\{\theta(a)(\omega\otimes\iota)(X) \mid a\in A, \omega\in \ell^\infty(\hat{\mathbb{G}})_* \} \subset M(B)
\end{equation*}
becomes a $C^*$-algebra. 
The \textit{reduced crossed product $C^*$-algebra} is the $C^*$-algebra associated with the covariant representation $((\hat{\lambda}\otimes\iota)\alpha, \hat{\mathscr{W}}\otimes 1)$ into ${\cal L}(L^2(\mathbb{G})\otimes A)$, where $\hat{\lambda}$ is the left regular representation, $\hat{\mathscr{W}}:=(\iota\otimes\rho)(\mathbb{V})$, and ${\cal L}(L^2(\mathbb{G})\otimes A)$ is the $C^*$-algebra of all right $A$-module maps on the Hilbert module $L^2(\mathbb{G})\otimes A$. 
The \textit{full crossed product $C^*$-algebra} is defined as a universal object for all covariant representations. 
When $A\subset \mathbb{B}(H)$ is a von Neumann algebra, the \textit{crossed product von Neumann algebra} is defined as the $\sigma$-weak closure of the reduced crossed product in $\mathbb{B}(L^2(\mathbb{G})\otimes H)$. 
We denote them by $\hat{\mathbb{G}}\ltimes_{\rm r}A$, $\hat{\mathbb{G}}\ltimes_{\rm f}A$ and $\hat{\mathbb{G}}\ltimes A$ respectively. 

The crossed product von Neumann algebra has a canonical conditional expectation $E_A$ onto $A$, given by 
\begin{equation*}
\hat{\mathbb{G}}\ltimes A \ni \theta(a)(\omega\otimes\iota\otimes\iota)(\hat{\mathscr{W}}\otimes 1) \mapsto a (\omega\otimes h \otimes\iota)(\hat{\mathscr{W}}\otimes 1)\in A,
\end{equation*}
where $h$ is the Haar state (more explicitly it sends $\theta(a)\rho(u_{i,j}^z)$ to $a h(u_{i,j}^z)$).
For a faithful normal state $\phi$ on $A$, we can define a canonical faithful normal state on $\hat{\mathbb{G}}\ltimes A$ by $\tilde{\phi}:=\phi\circ E_A$. When $\phi$ is a trace and the given action $\alpha$ is $\phi$-preserving (i.e.\ $(\iota\otimes\phi)\alpha(a)=\phi(a)$ for $a\in A$), $\tilde{\phi}$ also becomes a trace.

\bigskip

{\bf $\bullet$ Free products}

\bigskip

For free products of discrete quantum groups, we refer the reader to $\cite{Wa95}$.

We first recall fundamental facts of free products of $C^*$-algebras. Let $A_i$ $(i=1,\ldots,n)$ be unital $C^{\ast}$-algebras and $\phi_i$ non-degenerate states on $A_i$ (i.e.\ GNS representations of $\phi_i$ are faithful).
Denote GNS-representations of $\phi_i$ by $(\pi_{i},H_{i},\xi_{i})$ and decompose $H_{i}=\mathbb{C}\xi_{i}\oplus H_{i}^{0}$ as Hilbert spaces, where $H_{i}^{0}:=(\mathbb{C}\xi_{i})^{\perp}$. Define two Hilbert spaces by

\begin{eqnarray*}
H_1*\cdots*H_n:=\mathbb{C}\Omega \oplus \bigoplus_{k\geq1}\hspace{0.3em}&\bigoplus_{(i_1,\dots,i_k)\in N_k}& H_{i_{1}}^{0}\otimes H_{i_{2}}^{0}\cdots \otimes H_{i_{k}}^{0},\\
H(i):=\mathbb{C}\Omega \oplus \bigoplus_{k\geq1}\hspace{0.3em}&\bigoplus_{(i_1,\dots,i_k)\in N_k, i\neq i_{1}}& H_{i_{1}}^{0}\otimes H_{i_{2}}^{0}\cdots \otimes H_{i_{k}}^{0},
\end{eqnarray*}
where $\Omega$ is a fixed norm one vector and $N_k:=\{(i_1,\dots,i_k)\in \{1,\ldots,n\}^k \mid i_l\neq i_{l+1} \textrm{ for all }l=1,\dots,k-1\}$. Write $H:=H_1*\cdots*H_n$. Let $W_i$ be unitary operators given by
\begin{eqnarray*}
W_{i}\colon H
&=& H(i)\oplus(H_{i}^{0}\otimes H(i))\\
&\simeq& (\mathbb{C}\xi_{i}\otimes H(i))\oplus(H_{i}^{0}\otimes H(i))\\
&\simeq&(\mathbb{C}\xi_{i} \oplus H_{i}^{0})\otimes H(i)\\
&=&H_{i}\otimes H(i)
\end{eqnarray*}
Then a canonical $A_i$-action (more generally $\mathbb{B}(H_i)$-action) on $H$ is given by $\lambda_i(a):=W_i^* (a\otimes 1)W_i$ $(a\in A_i)$. The \textit{free product $C^*$-algebra} $(A_1,\phi_1)*\cdots *(A_n,\phi_n)$ is the $C^*$-algebra generated by $\lambda_i(A_i)$  $(i=1,\ldots, n)$. 
The vector state of the canonical cyclic vector $\Omega$ is called the free product state and denoted by $\phi_1*\cdots *\phi_n$. 
 When each $A_i$ is a von Neumann algebra and $\phi_i$ is normal, the \textit{free product von Neumann algebra} is defined as the $\sigma$-weak closure of the free product $C^*$-algebra in $\mathbb{B}(H)$. We denote it by $(A_1,\phi_1)\bar{*}\cdots \bar{*}(A_n,\phi_n)$.

Let $M_i$ $(i=1,\ldots,n)$ be von Neumann algebras and $\phi_i$ be faithful normal states on $M_i$. Denote modular objects of them by $\Delta_i$ and $J_i$. Then modular objects on the free product von Neumann algebra are given by
\begin{alignat*}{3}
\Delta^{it} &\colon H_{i_{1}}^{0}\otimes \cdots \otimes H_{i_{k}}^{0} \ni \xi_{i_{1}}\otimes \cdots \otimes \xi_{i_{k}} \longmapsto \Delta^{it}_{i_1}\xi_{i_{1}}\otimes \cdots \otimes \Delta^{it}_{i_k}\xi_{i_{k}}  &\in& H_{i_{1}}^{0}\otimes \cdots \otimes H_{i_{k}}^{0},\\
J \hspace{0.7em} &\colon H_{i_{1}}^{0}\otimes \cdots \otimes H_{i_{k}}^{0} \ni \xi_{i_{1}}\otimes \cdots \otimes \xi_{i_{k}} \longmapsto J_{i_k}\xi_{i_{k}}\otimes \cdots \otimes J_{i_1}\xi_{i_{1}}  &\in& H_{i_{k}}^{0}\otimes \cdots \otimes H_{i_{1}}^{0}.
\end{alignat*}
The unitary elements $V_i:=\Sigma(J_i\otimes J|_{JH(i)})W_iJ$, where $\Sigma$ is a flip from $H_i\otimes H(i)$ to $H(i)\otimes H_i$, gives a right $M_i$-action (more generally $\mathbb{B}(H_i)$-action) by $\rho_i(a):=J\lambda_i(a)J=V_i^*(1\otimes J_iaJ_i)V_i$.

Let $\mathbb{G}_i$ $(i=1,\ldots,n)$ be compact quantum groups and $h_i$ Haar states of them. Then the reduced free product $(C_{\rm red}(\mathbb{G}_1), h_1)*\cdots *(C_{\rm red}(\mathbb{G}_n), h_n)$ has a structure of compact quantum group with the Haar state $h_1*\cdots *h_n$. 
We write its dual as $\hat{\mathbb{G}}_1*\cdots *\hat{\mathbb{G}}_n$ and call it the \textit{free product} of $\hat{\mathbb{G}}_i$ $(i=1,\ldots,n)$. 
Modular objects $\hat{J}$ and $U:=J\hat{J}$ are given by
\begin{alignat*}{3}
\hat{J} \hspace{0.5em} &\colon H_{i_{1}}^{0}\otimes \cdots \otimes H_{i_{n}}^{0} \ni \xi_{i_{1}}\otimes \cdots \otimes \xi_{i_{n}} \longmapsto \hat{J}_{i_1}\xi_{i_1}\otimes \cdots \otimes \hat{J}_{i_n}\xi_{i_{n}}  &\in& H_{i_{1}}^{0}\otimes \cdots \otimes H_{i_{n}}^{0},\\
U \hspace{0.5em} &\colon H_{i_{1}}^{0}\otimes \cdots \otimes H_{i_{n}}^{0} \ni \xi_{i_{1}}\otimes \cdots \otimes \xi_{i_{n}} \longmapsto U_{i_n}\xi_{i_{n}}\otimes \cdots \otimes U_{i_1}\xi_{i_{1}}  &\in& H_{i_{n}}^{0}\otimes \cdots \otimes H_{i_{1}}^{0}.
\end{alignat*}
We have a formula $W_i\hat{J}=(\hat{J}_i\otimes \hat{J}|_{H(i)}) W_i$.

By the natural inclusion $(C_{\rm red}(\mathbb{G}_i), h_i)\subset (C_{\rm red}(\mathbb{G}_1), h_1)*\cdots *(C_{\rm red}(\mathbb{G}_n), h_n)$ $(i=1,\ldots,n)$, we can regard all corepresentations of $\mathbb{G}_i$ as those of the dual of  $\hat{\mathbb{G}}_1*\cdots *\hat{\mathbb{G}}_n$. 
A representative of all irreducible representations of the dual of  $\hat{\mathbb{G}}_1*\cdots *\hat{\mathbb{G}}_n$ is given by the trivial corepresentation and all corepresentations of the form $w_1\boxtimes w_2\boxtimes \cdots \boxtimes w_n$, where each $w_l$ is in $\mathrm{Irred}(\mathbb{G}_i)$ for some $i$ (and not trivial corepresentation) and no two adjacent factors are taken from corepresentations of the same quantum group.

\bigskip

{\bf $\bullet$ Amenability of actions}

\bigskip

Let $\alpha$ be an action of a discrete quantum $\hat{\mathbb{G}}$ on a unital $C^*$-algebra $A$. Denote by $L^2(\mathbb{G})\otimes A$ the Hilbert module obtained from $L^2(\mathbb{G})\otimes_{\rm alg} A$ with $\langle\xi\otimes a | \eta\otimes b\rangle:=\langle\xi|\eta\rangle b^*a$ for $\xi,\eta\in L^2(\mathbb{G})$ and $a,b\in A$. 
We say $\alpha$ is \textit{amenable} $\cite[\textrm{Definition 4.1}]{VV05}$ if 
there exists a sequence $\xi_n\in L^2(\mathbb{G})\otimes A$ satisfying
\begin{itemize}
	\item[$\rm (i)$] $\langle\xi_n |\xi_n\rangle \rightarrow 1$ in $A$;
	\item[$\rm (ii)$] for all $x\in \mathrm{Irred}(\mathbb{G})$, $\|((\iota\otimes \alpha)(\xi_n)-\hat{\mathscr{V}_{12}}(\xi_n)_{13}) (1\otimes p_x \otimes 1)\|\rightarrow 0$;
	\item[$\rm (iii)$] $((\hat{\rho}\times\hat{\lambda})\circ\hat{\Phi}\otimes \iota)\alpha(a)\xi_n =\xi_n a$ for any $n\in\mathbb{N}$ and $a\in A$,
\end{itemize}
where $\hat{\mathscr V}:=(\lambda\otimes \iota)(\mathbb{V}_{21})$ and
$\hat{\rho}\times\hat{\lambda}$ is the multiplication map from $\ell^\infty(\hat{\mathbb{G}})\otimes\ell^\infty(\hat{\mathbb{G}})$ into $\mathbb{B}(L^2(\mathbb{G}))$, which is bounded since $\ell^\infty(\hat{\mathbb{G}})$ is amenable. 
It is clear that $\hat{\mathbb{G}}$ is amenable if and only if any trivial action of $\hat{\mathbb{G}}$ is amenable.

\bigskip

{\bf $\bullet$ Free quantum groups and quantum automorphism groups}

\bigskip

Let $F$ be a matrix in $\mathrm{GL}(n,\mathbb{C})$. The \textit{free unitary quantum group} (resp.\ \textit{free orthogonal quantum group}) of $F$ $\cite{VW96}\cite{Wa95}$ is the $C^*$-algebra $C(A_u(F))$ (resp.\ $C(A_o(F))$) is defined as the universal unital $C^*$-algebra generated by all the entries of a unitary $n$ by $n$ matrix $u=(u_{i,j})_{i,j}$ satisfying that $F(u_{i,j}^*)_{i,j}F^{-1}$ is unitary (resp.\ $F(u_{i,j}^*)_{i,j}F^{-1}=u$).

Recall that a \textit{coaction} of a compact quantum group $\mathbb{G}$ on a unital $C^*$-algebra $B$ is a unital $*$-homomorphism $\beta\colon B\rightarrow B\otimes C(\mathbb{G})$ satisfying $(\beta\otimes \iota)\circ \beta=(\iota\otimes \Phi)\circ \beta$ and $\overline{{\rm span}}\{\beta(B)(1\otimes A)=B\otimes C(\mathbb{G})$.
Let $(B,\phi)$ be a pair of a finite dimensional $C^*$-algebra and a faithful state on $B$. Then the \textit{quantum automorphism group} of $(B,\phi)$ $\cite{Wa98_1}\cite{Ba01}$ is the universal compact quantum group $A_{\rm aut}(B,\phi)$ which can be endowed with a coaction $\beta\colon B\rightarrow B\otimes C(A_{\rm aut}(B,\phi))$ satisfying $(\phi\otimes \iota)\circ \beta(x)=\phi(x)1$ for $x\in B$.

\subsection{\bf Weak amenability and the $\rm \bf W^*$CBAP}

Let $\hat{\mathbb{G}}$ be a discrete quantum group. Denote the dense Hopf $*$-algebra by ${\mathscr C}(\hat{\mathbb{G}})$. 
For any element $a\in\ell^\infty(\hat{\mathbb{G}})$, we can associate a linear map $m_a$ from ${\mathscr C}(\hat{\mathbb{G}})$ to ${\mathscr C}(\hat{\mathbb{G}})$, given by $(m_a\otimes \iota)(u^x)=(1\otimes ap_x)u^x$ for any $x\in \mathrm{Irred}(\mathbb{G})$, where $p_x\in c_0(\hat{\mathbb{G}})$ is the canonical projection onto $x$ component. 
We we say $\hat{\mathbb{G}}$ is \textit{weakly amenable} if there exist a net $(a_i)_i$ of elements of $\ell^\infty(\hat{\mathbb{G}})$ such that  
\begin{itemize}
	\item each $a_i$ has finite support, namely, $a_ip_x=0$ except for finitely many $x\in \mathrm{Irred}(\mathbb{G})$;
	\item $(a_i)_i$ converges to 1 pointwise, namely, $a_ip_x$ converges to $p_x$ in $\mathbb{B}(H_x)$ for any $x\in \mathrm{Irred}(\mathbb{G})$;
	\item $\limsup_i\|m_{a_i}\|_{\rm c.b.}$ is finite.
\end{itemize}

We also recall that a von Neumann algebra $M$ has the $\it weak^*$ \textit{completely approximation property} (or $\ W^*$\textit{CBAP}, in short) 
if there exists a net $(\psi_i)_i$ of completely bounded (say c.b.) maps on $M$ with normal and finite rank such that $\limsup_i\|\psi_i\|_{\rm c.b.}<\infty$ and $\psi_i$ converges to $\mathrm{id}_M$ in the point $\sigma$-weak topology.

Then the \textit{Cowling--Haagerup constant} of $\hat{\mathbb{G}}$ and $M$ are defined as 
\begin{eqnarray*}
&&\Lambda_{\rm c.b.}(\hat{\mathbb{G}}):=\inf\{\ \limsup_i\|m_{a_i}\|_{\rm c.b.}\mid (a_i)_i \textrm{ satisfies the above condition}\},\\
&&\Lambda_{\rm c.b.}(M):=\inf\{\ \limsup_i\|\psi_i\|_{\rm c.b.}\mid (\psi_i) \textrm{ satisfies the above condition}\}.
\end{eqnarray*}
It is known that $\Lambda_{\rm c.b.}(\hat{\mathbb{G}})\geq\Lambda_{\rm c.b.}(L^\infty(\mathbb{G}))$.

De Commer, Freslon, and Yamashita recently proved that $\Lambda_{\rm c.b.}(\hat{\mathbb{G}})=1$, where $\mathbb{G}$ is a free quantum group or a quantum automorphism group $\cite{DFY13}$. 
Note that a special case of this was already solved by Freslon $\cite{Fr12}$.

\subsection{\bf Popa's intertwining techniques and relative amenability}

We first recall Popa's intertwining techniques in both non-finite and semifinite situations. 

\begin{Def}\label{popa embed def}\upshape
Let $M$ be a von Neumann algebra, $p$ and $q$ projections in $M$, $A\subset qMq$ and $B\subset pMp$ von Neumann subalgebras, and let $E_{B}$ be a faithful normal conditional expectation from $pMp$ onto $B$. Assume $A$ and $B$ are finite. We say $A$ \textit{embeds in $B$ inside} $M$ and denote by $A\preceq_M B$ if there exist non-zero projections $e\in A$ and $f\in B$, a unital normal $\ast$-homomorphism $\theta \colon eAe \rightarrow fBf$, and a partial isometry $v\in M$ such that
\begin{itemize}
\item $vv^*\leq e$ and $v^*v\leq f$,
\item $v\theta(x)=xv$ for any $x\in eAe$.
\end{itemize}
\end{Def}

\begin{Thm}[\textit{non-finite version, }{\cite{Po01}\cite{Po03}\cite{Ue12}\cite{HV12}}]\label{popa embed}
Let $M,p,q,A,B$, and $E_{B}$ be as in the definition above, and let $\tau$ be a faithful normal trace on $B$. Then the following conditions are equivalent.
\begin{itemize}
	\item[$\rm (i)$]The algebra $A$ embeds in $B$ inside $M$.
	\item[$\rm (ii)$]There exists no sequence $(w_n)_n$ of unitaries in $A$ such that  $E_{B}(b^*w_n a)\rightarrow 0$ strongly for any $a,b\in qMp$.
\end{itemize}
\end{Thm}

\begin{Thm}[\textit{semifinite version, }{\cite{CH08}\cite{HR10}}]\label{popa embed2}
Let $M$ be a semifinite von Neumann algebra with a faithful normal semifinite trace $\mathrm{Tr}$, and $B\subset M$ be a von Neumann subalgebra with $\mathrm{Tr}_{B}:=\mathrm{Tr}|_{B}$ semifinite. Denote by $E_{B}$  the unique $\mathrm{Tr}$-preserving conditional expectation from $M$ onto $B$. 
Let $q$ be a $\mathrm{Tr}$-finite projection in $M$ and $A\subset qMq$ a von Neumann subalgebra. Then the following conditions are equivalent.
\begin{itemize}
	\item[$\rm (i)$]There exists a non-zero projection $p\in B$ with $\mathrm{Tr}_{B}(p)<\infty$ such that $A\preceq_{eMe}pBp$, where $e:=p\vee q$.
	\item[$\rm (ii)$]There exists no sequence $(w_n)_n$ in unitaries of $A$ such that $E_{B}(b^{\ast}w_n a)\rightarrow 0$ strongly for any $a,b\in qM$.
\end{itemize}
We use the same symbol $A\preceq_{M}B$ if one of these conditions holds.
\end{Thm}
\noindent
By the proof of the semifinite one, we have that $A\not\preceq_M B$ if and only if there exists a net $(p_j)$ of $\mathrm{Tr}_B$-finite projections in $B$ which converges to 1 strongly and $A\not\preceq_{e_jMe_j}p_jBp_j$, where $e_j:=p_j\vee q$. 
We also mention that when $A$ is diffuse and $B$ is atomic, then $A\not\preceq_M B$. This follows from the existence of a normal unital $*$-homomorphism $\theta$ in the definition.

We next recall relative amenability introduced in $\cite{OP07}$ and $\cite{PV11}$.

\begin{Def}\upshape
Let $(M,\tau)$ be a tracial von Neumann algebra, $q\in M$ a projection, and $B\subset M$ and $A\subset qMq$ be von Neumann subalgebras. 
We say $A$ is amenable relative to $B$ in $M$ if there exists a state $\phi$ on $q\langle M,e_B\rangle q$ such that  $\phi$ is $A$-central and the restriction of $\phi$ on $A$ coincides with $\tau$.
\end{Def}

\begin{Def}\upshape
Let $(M,\tau)$ and $(B,\tau_B) $ be tracial von Neumann algebras and $A\subset M$ be a von Neumann subalgebra. 
We say an $M$-$B$-bimodule ${}_M K_B$ is left $A$-amenable if there exists a state $\phi$ on $\mathbb{B}(K)\cap (B^{\rm op})'$ such that  $\phi$ is $A$-central and the restriction of $\phi$ on $A$ coincides with $\tau$.
\end{Def}

We note that for any $B\subset M$ and $A\subset qMq$, amenability of $A$ relative to $B$ in $M$ is equivalent to left $A$-amenability of $qL^2(M)$ as a $qMq$-B-bimodule, since $q\langle M,e_B\rangle q= q(\mathbb{B}(L^2(M))\cap (B^{\rm op})')q=\mathbb{B}(qL^2(M))\cap (B^{\rm op})'$. 
We also mention that when $B$ is amenable, then since $\mathbb{B}(K)\cap (B^{\rm op})'$ is amenable, there exists a conditional expectation from $\mathbb{B}(K)$ onto $\mathbb{B}(K)\cap (B^{\rm op})'$. 
In this case, relative amenability of $A$ (or left $A$-amenability) means amenability of $A$.

\section{\bf Bi-exactness}

\subsection{\bf Two definitions of bi-exactness}

We introduce two notions of bi-exactness on discrete quantum groups. These notions are equivalent for discrete groups as we have seen in Introduction. 
Recall that $C_{\rm red}(\mathbb{G})=\rho(C(\mathbb{G}))$ and $UC_{\rm red}(\mathbb{G})U=\lambda(C(\mathbb{G}))=C_{\rm red}(\mathbb{G})^{\rm op}$, where $U=J\hat{J}=\hat{J}J$. Basically we use $UC_{\rm red}(\mathbb{G})U$ instead of $C_{\rm red}(\mathbb{G})^{\rm op}$.

\begin{Def}\upshape\label{bi-ex}
Let $\hat{\mathbb{G}}$ be a discrete quantum group. We say $\hat{\mathbb{G}}$ is \textit{bi-exact} if it satisfies following conditions:
\begin{itemize}
	\item[$\rm (i)$] the quantum group $\hat{\mathbb{G}}$ is exact (i.e.\ $C_{\rm red}(\mathbb{G})$ is exact);
	\item[$\rm (ii)$]  there exists a u.c.p.\ map $\theta\colon C_{\rm red}(\mathbb{G})\otimes  C_{\rm red}(\mathbb{G}) \rightarrow \mathbb{B}(L^2(\mathbb{G}))$ such that $\theta(a\otimes b)-aUbU\in \mathbb{K}(L^2(\mathbb{G}))$ for any $a,b \in  C_{\rm red}(\mathbb{G})$.
\end{itemize}
\end{Def}
\begin{Def}\upshape\label{st bi-ex}
Let $\hat{\mathbb{G}}$ be a discrete quantum group. We say $\hat{\mathbb{G}}$ is \textit{strongly bi-exact} if there exists a unital $C^*$-subalgebra ${\cal B}$ in $\ell^{\infty}(\hat{\mathbb{G}})$ such that  
\begin{itemize}
	\item[$\rm (i)$] the algebra $\cal B$ contains $c_0(\hat{\mathbb{G}})$, and the quotient ${\cal B}_{\infty}:={\cal B}/c
_0(\hat{\mathbb{G}})$ is nuclear;
	\item[\rm (ii)] the left translation action on $\ell^\infty(\hat{\mathbb{G}})$ induces an amenable action on ${\cal B}_{\infty}$, and the right one induces the trivial action on ${\cal B}_{\infty}$.
\end{itemize}
\end{Def}

\begin{Rem}\upshape
Amenability implies strong bi-exactness. In fact, we can choose ${\cal B}:=c_0(\hat{\mathbb{G}})+\mathbb{C}1$. 
In this case, both the left and right actions on ${\cal B}_\infty(\simeq \mathbb{C})$ are trivial.
\end{Rem}

We first observe relationship between bi-exactness and strong bi-exactness. 
In (i) of the definition of strong bi-exactness, nuclearity of ${\cal B}_\infty$ is equivalent to that of $\cal B$. Moreover, together with the condition (ii), the $C^*$-subalgebra ${\cal C}_l\subset \mathbb{B}(L^2(\mathbb{G}))$ generated by $\hat{\lambda}({\cal B})$ and $C_{\rm red}(\mathbb{G})(=\rho(C(\mathbb{G})))$ is also nuclear. 
In fact, the quotient image of ${\cal C}_l$ in $\mathbb{B}(L^2(\mathbb{G}))/\mathbb{K}(L^2(\mathbb{G}))$ is nuclear since there is a canonical surjective map from $\hat{\mathbb{G}} \ltimes_{\rm f} {\cal B}_\infty$ to ${\cal C}_l$ and $\hat{\mathbb{G}} \ltimes_{\rm f} {\cal B}_\infty$ is nuclear by amenability of the action. 
Then ${\cal C}_l$ is an extension of $\mathbb{K}(L^2(\mathbb{G}))$ by this quotient image, and hence is nuclear. 
Note that ${\cal C}_l$ contains $\mathbb{K}(L^2(\mathbb{G}))$, since it contains $C_{\rm red}(\mathbb{G})$ and the orthogonal projection from $L^2(\mathbb{G})$ onto $\mathbb{C}\hat{1}$. 
We put ${\cal C}_r:=U{\cal C}_lU$, where $U:=J\hat{J}$. Triviality of the right action in (ii) implies that all commutators of $UC_{\rm red}(\mathbb{G})U$ and $\hat{\lambda}({\cal B})$ (respectively $C_{\rm red}(\mathbb{G})$ and $U\hat{\lambda}({\cal B})U$) are contained in $\mathbb{K}(L^2(\mathbb{G}))$. 
This implies that all commutators of ${\cal C}_l$ and ${\cal C}_r$ are also contained in $\mathbb{K}(L^2(\mathbb{G}))$.
Thus we obtained the following $\ast$-homomorphism:

\begin{equation*}
\nu\colon {\cal C}_l \otimes {\cal C}_r \longrightarrow \mathbb{B}(L^2(\mathbb{G}))/\mathbb{K}(L^2(\mathbb{G})); a\otimes b \longmapsto ab.
\end{equation*}
This map is an extension of the multiplication map on $C_{\rm red}(\mathbb{G}) \otimes UC_{\rm red}(\mathbb{G})U$, and so this multiplication map is nuclear since so is ${\cal C}_l \otimes {\cal C}_r$. Finally by the lifting theorem of Choi and Effros $\cite{CE76}$ (or see $\cite[\textrm{Theorem C.3}]{BO08}$), we obtain a u.c.p.\ lift $\theta$ of the multiplication map. 
Thus we observed that strong bi-exactness implies bi-exactness (exactness of $\hat{\mathbb{G}}$ follows from nuclearity of ${\cal C}_l$). The intermediate object ${\cal C}_l$ is important for us, and we will use this algebra in the next subsection. 
We summary these observations as follows.

\begin{Lem}\label{intermediate}
Strong bi-exactness implies bi-exactness. 
The following condition is an intermediate condition between bi-exactness and strong bi-exactness:
\begin{itemize}
	\item There exists a nuclear $C^*$-algebra ${\cal C}_l\subset \mathbb{B}(L^2(\mathbb{G}))$ which contains $C_{\rm red}(\mathbb{G})$ and $\mathbb{K}(L^2(\mathbb{G}))$, and all commutators of ${\cal C}_l$ and ${\cal C}_r(:=U{\cal C}_lU)$ are contained in $\mathbb{K}(L^2(\mathbb{G}))$.
\end{itemize}
\end{Lem}

Examples of bi-exact quantum groups were first given by Vergnioux $\cite{Ve05}$. He constructed a u.c.p.\ lift directly for free quantum groups. Then he, in a joint work with Vaes $\cite{VV05}$, gave a new proof for bi-exactness of $\hat{A}_o(F)$ and they in fact proved strong bi-exactness. 
In the proof, they only used the fact that $A_o(F)$ is monoidally equivalent to some $\mathrm{SU}_q(2)$ with $-1<q<1$ and $q\neq0$, seeing some estimates on intertwiner spaces of $\mathrm{SU}_q(2)$ $\cite[\textrm{Lemma 8.1}]{VV05}$. 
Since the dual of $\mathrm{SO}_q(3)$ is a quantum subgroup of some dual of $\mathrm{SU}_q(2)$, intertwiner spaces of $\mathrm{SO}_q(3)$ have the same estimates. 
From this fact, we can deduce strong bi-exactness of a dual of a compact quantum group which is monoidally equivalent to $\mathrm{SO}_q(3)$ (by the same argument as that for $\mathrm{SU}_q(2)$). 
We also mention that strong bi-exactness of $A_u(F)$ was proved by the same argument $\cite{VV08}$. 

We summary these observations as follows.
\begin{Thm}\label{example st bi-exact}
Let $\mathbb{G}$ be a compact quantum group which is monoidally equivalent to $\mathrm{SU}_q(2)$, $\mathrm{SO}_q(3)$, or $A_u(F)$, where $-1<q<1$, $q\neq 0$, $F$ is not a scalar multiple of a $2$ by $2$ unitary. Then the dual $\hat{\mathbb{G}}$ is strongly bi-exact.
\end{Thm}

In $\cite{Is12_2}$, we introduced condition $\rm (AOC)^+$, which is similar to condition $\rm (AO)^+$ on continuous cores, and proved that von Neumann algebras of free quantum groups satisfy this condition. 
In the proof we also used only the fact that $A_o(F)$ is monoidally equivalent to some $\mathrm{SU}_q(2)$ and hence we actually proved the following statement. 

\begin{Thm}\label{example bi-exact}
Let $\mathbb{G}$ be a compact quantum group which is monoidally equivalent to $\mathrm{SU}_q(2)$, $\mathrm{SO}_q(3)$, or $A_u(F)$, where $-1<q<1$, $q\neq 0$, $F$ is not a scalar multiple of a $2$ by $2$ unitary. Then $L^\infty(\mathbb{G})$ and its Haar state satisfy condition $\rm (AOC)^+$ with the dense $C^*$-algebra $C_{\rm red}(\mathbb{G})$.
\end{Thm}

In the proof we gave a sufficient condition to condition $\rm (AOC)^+$, which was formulated for general von Neumann algebras $\cite[\textrm{Lemma 3.2.3}]{Is12_2}$. 
When we see a quantum group von Neumann algebra, we have a more concrete sufficient condition as follows. To verify this, see Subsection 3.2 in the same paper. 
In the lemma below, $\pi$ means the canonical $*$-homomorphism from $\mathbb{B}(L^2(\mathbb{G}))$ into $\mathbb{B}(L^2(\mathbb{G})\otimes L^2(\mathbb{R}))$ defined by $(\pi(x)\xi)(t):=\Delta_h^{-it}x\Delta_h^{it}\xi(t)$ for $x\in\mathbb{B}(L^2(\mathbb{G}))$, $t\in\mathbb{R}$, and $\xi\in L^2(\mathbb{G})\otimes L^2(\mathbb{R})$. 

\begin{Lem}\label{AOC}
Let $\mathbb{G}$ be a compact quantum group and ${\cal C}_l\subset C^*\{C_{\rm red}(\mathbb{G}), \hat{\lambda}(\ell^\infty(\hat{\mathbb{G}})) \}$ a $C^*$-subalgebra which contains $C_{\rm red}(\mathbb{G})$ and $\mathbb{K}(L^2(\mathbb{G}))$. Put ${\cal C}_r:=U{\cal C}_lU$. Assume the following conditions:
\begin{itemize}
	\item[$\rm (i)$] the algebra ${\cal C}_l$ is nuclear;
	\item[$\rm (ii)$] a family of maps $\mathrm{Ad}\Delta_h^{it}$ $(t\in\mathbb{R})$ gives a norm continuous action of $\mathbb{R}$ on ${\cal C}_l$;
	\item[$\rm (iii)$] all commutators of $\pi({\cal C}_l)$ and ${\cal C}_r \otimes1$ are contained in $\mathbb{K}(L^2(\mathbb{G}))\otimes\mathbb{B}(L^2(\mathbb{R}))$.
\end{itemize}
Then $L^\infty(\mathbb{G})$ and its Haar state satisfy condition $\rm (AOC)^+$ with the dense $C^*$-algebra $C_{\rm red}(\mathbb{G})$.
\end{Lem}

\begin{Rem}\upshape\label{amenable}
When $\hat{\mathbb{G}}$ is amenable, then $L^\infty(\mathbb{G})$ and its Haar state satisfy condition $\rm (AOC)^+$. In fact, we can choose ${\cal B}:=c_0(\hat{\mathbb{G}})+\mathbb{C}1$ and ${\cal C}_l:=C^*\{C_{\rm red}(\mathbb{G}), \hat{\lambda}({\cal B})\}$. 
In this case, all conditions in this lemma are easily verified.
\end{Rem}

\begin{Rem}\upshape\label{group}
When $\hat{\mathbb{G}}$ is a strongly bi-exact discrete quantum group of Kac type (possibly bi-exact discrete group) with ${\cal B}\subset \ell^\infty(\hat{\mathbb{G}})$, then since the modular operator is trivial, ${\cal C}_l:=C^*\{C_{\rm red}(\mathbb{G}), \lambda({\cal B})\}$ satisfies these conditions.
\end{Rem}

\begin{Rem}\upshape
In the proof of $\cite[\textrm{Proposition 3.2.4}]{Is12_2}$, we put ${\cal C}_l=\hat{\mathbb{G}}\ltimes_{\rm r} {\cal B}_\infty$ (here we write it as $\tilde{\cal C}_l$), and in this subsection we are putting ${\cal C}_l=C^*\{C_{\rm red}(\mathbb{G}), \hat{\lambda}({\cal B})\}$. 
Both of them are nuclear $C^*$-algebras containing $C_{\rm red}(\mathbb{G})$ and do the same work to get condition $\rm (AOC)^+$.
The difference of them is that ${\cal C}_l$ is contained in $\mathbb{B}(L^2(\mathbb{G}))$ but $\tilde{\cal C}_l$ is not. 
Hence ${\cal C}_l$ is more useful and $\tilde{\cal C}_l$ is more general (since $\tilde{\cal C}_l$ produces ${\cal C}_l$). In the previous paper, we preferred the generality and hence used $\tilde{\cal C}_l$ in the proof.
\end{Rem}

\subsection{\bf Free products of bi-exact quantum groups}

Free products of bi-exact discrete groups (more generally, free products of von Neumann algebras with condition (AO)) were studied in $\cite{Oz04}\cite[\textrm{Section 4}]{GJ07}\cite[\textrm{Section 15.3}]{BO08}$. 
In this subsection we will prove similar results on discrete quantum groups. We basically follow the strategy in $\cite{Oz04}$.

\begin{Lem}[{\cite[\rm Lemma\ 2.4]{Oz04}}]\label{nuclear}
Let $B_i\subset \mathbb{B}(H_i)$ $(i=1,2)$ be $C^*$-subalgebras with $B_i$-cyclic vectors $\xi_i$ and denote by $\omega_i$ the corresponding vector states (note that each $\omega_i$ is non-degenerate). If each $B_i$ contains $P_i$, the orthogonal projection onto $\mathbb{C}\xi_i$, and is nuclear, then the free product $(B_1,\omega_1)*(B_2,\omega_2)$ is also nuclear.
\end{Lem}
\begin{Rem}\upshape
In this Lemma, each $B_i$ contains $\mathbb{K}(H_i)$ since it contains $P_i$ and the vector $\xi_i$ is $B_i$-cyclic. 
Projections $\lambda_i(P_i)\in(B_1,\omega_1)*(B_2,\omega_2)$ are orthogonal projections onto $H(i)$ and hence the orthogonal projection onto $\mathbb{C}\Omega$ is of the form $\lambda_1(P_1)+\lambda_2(P_2)-1$, which is contained in $(B_1,\omega_1)*(B_2,\omega_2)$. 
Since the vector $\Omega$ is cyclic for $(B_1,\omega_1)*(B_2,\omega_2)$, $(B_1,\omega_1)*(B_2,\omega_2)$ contains all compact operators.
\end{Rem}

For free product von Neumann algebras, we use the same notation  $W_i$, $V_i$, $\Delta$, $\Delta_i$, $J$ and $J_i$ as in the free product part of Subsection \ref{CQG}.

\begin{Lem}[{\cite[\rm Lemma\ 3.1]{Oz04}}]
For a free product von Neumann algebra $(M_1,\phi_1)*\cdots*(M_n,\phi_n)$, we have the following equations:
\begin{eqnarray*}
\lambda_i(a) = J\rho_i(a)J
&=&JV_i^*(1_{JH(i)}\otimes J_iaJ_i)V_iJ \\
&=&V_i^*(P_\Omega\otimes a+\lambda_i(a)\mid_{JH(i)\ominus \mathbb{C}\Omega}\otimes 1_{H_i})V_i\\
&=&V_j^*(\lambda_i(a)\mid_{JH(j)}\otimes 1_{H_j})V_j,
\end{eqnarray*}
for any $a\in \mathbb{B}(H_i)$ and $i\neq j$, where $P_\Omega$ is the orthogonal projection onto $\mathbb{C}\Omega$. 
\end{Lem}
\begin{Rem}\upshape\label{commutator}
Simple calculations with the lemma show that $[\lambda_i(\mathbb{B}(H_i)),J\lambda_j(\mathbb{B}(H_j))J]=0$ when $i\neq j$, and that
\begin{equation*}
[\lambda_i(a), J\lambda_i(b)J]= V_i^* (P_\Omega \otimes [a, J_ibJ_i])V_i
\end{equation*}
for $a,b\in \mathbb{B}(H_i)$. 
Since $V_i=\Sigma(J_i\otimes J|_{JH(i)})W_iJ$, where $\Sigma$ is the flip, this equation means 
\begin{eqnarray*}
[\lambda_i(a), J\lambda_i(b)J]
&=& V_i^* (P_\Omega \otimes [a, J_ibJ_i])V_i \\
&=&J^* W_i^* (J_i\otimes J|_{JH(i)})^*\Sigma^* (P_\Omega \otimes [a, J_ibJ_i]) \Sigma(J_i\otimes J|_{JH(i)})W_iJ \\
&=& J^* W_i^*(J_i[a, J_ibJ_i]J_i\otimes P_\Omega) W_iJ.
\end{eqnarray*}
Hence we get
\begin{equation*}
[\lambda_i(a), J\lambda_i(b)J]
=W_i^*([a, J_ibJ_i]\otimes P_\Omega) W_i \quad (a,b\in \mathbb{B}(H_i)).
\end{equation*}
This means the operator $[\lambda_i(a), J\lambda_i(b)J]$ is, as an operator on $H_1*\cdots *H_n$, $[a, J_ibJ_i]$ on $\mathbb{C}\Omega\oplus H_i^0(=H_i)$ and 0 otherwise.
\end{Rem}

\begin{Pro}\label{free prod bi-exact}
Let $\mathbb{G}_i$ $(i=1,\ldots,n)$ be compact quantum groups. If each $\hat{\mathbb{G}}_i$ satisfies the intermediate condition in \textrm{Lemma $\ref{intermediate}$} with ${\cal C}_l^i$, then the free product $\hat{\mathbb{G}}_1*\cdots*\hat{\mathbb{G}}_n$ satisfies the same condition with the nuclear $C^*$-algebra $({\cal C}_l^1,h_1)*\cdots* ({\cal C}_l^n,h_n)$, where $h_i$ are the vector states of $\hat{1}_{\mathbb{G}_i}$. 
In particular, $\hat{\mathbb{G}}_1*\cdots *\hat{\mathbb{G}}_n$ is bi-exact if each $\hat{\mathbb{G}}_i$ is strongly bi-exact.
\end{Pro}
\begin{proof}
%
We may assume $n=2$. Write $H:=L^2(\mathbb{G}_1)*L^2(\mathbb{G}_2)$. By Lemma \ref{nuclear} and the following remark, ${\cal C}_l^1* {\cal C}_l^2$ is nuclear and contains $\mathbb{K}(H)$.
So what to show is that commutators of ${\cal C}_l^1* {\cal C}_l^2$ and $U({\cal C}_l^1* {\cal C}_l^2)U$ are contained in $\mathbb{K}(H)$. 
We have only to check that $[\lambda_i({\cal C}_l^i),U\lambda_j({\cal C}_l^j)U]$ $(i,j=1,2)$ are contained in $\mathbb{K}(H)$, since $\mathbb{K}(H)$ is an ideal. 
This is easily verified by Remark \ref{commutator}.
\end{proof}

\begin{Pro}\label{free prod AOC^+}
Let $\mathbb{G}_i$ $(i=1,\ldots,n)$ be compact quantum groups. If each $\hat{\mathbb{G}}_i$ satisfies conditions in \textrm{Lemma $\ref{AOC}$} with ${\cal C}_l^i$, then the free product $\hat{\mathbb{G}}_1*\cdots*\hat{\mathbb{G}}_n$ satisfies the same condition with the nuclear $C^*$-algebra $({\cal C}_l^1,h_1)*\cdots* ({\cal C}_l^n,h_n)$, where $h_i$ are the vector states of $\hat{1}_{\mathbb{G}_i}$. 
\end{Pro}
\begin{proof}
We may assume $n=2$ and write $H:=L^2(\mathbb{G}_1)*L^2(\mathbb{G}_2)$. 
By the same manner as in the last proposition, ${\cal C}_l^1* {\cal C}_l^2$ is nuclear and contains $\mathbb{K}(H)$. 
This algebra is contained in $C^*\{C_{\rm red}(\mathbb{G}_1)*C_{\rm red}(\mathbb{G}_2), \hat{\lambda}(\ell^\infty(\hat{\mathbb{G}}_1*\hat{\mathbb{G}}_2)) \}$. 
Norm continuity of the modular action is trivial since it is continuous on each $\lambda_k({\cal C}_l^k)$. 
By Remark \ref{commutator}, our commutators in the algebra $\mathbb{B}(H)$ (not the algebra $\mathbb{B}(H\otimes L^2(\mathbb{R}))$) are of the form 
\begin{equation*}
[\lambda_k(a), U\lambda_k(b)U]
=W_k^*([a, U_kbU_k]\otimes P_\Omega) W_k \quad (a,b\in {\cal C}^k_l)
\end{equation*}
(or $[\lambda_k(a), U\lambda_l(b)U]=0$ when $k\neq l$). Modular actions for $a$  is of the form
\begin{equation*}
[\Delta^{it}\lambda_k(a)\Delta^{-it}, U\lambda_k(b)U]
=[\lambda_k(\Delta_k^{it}a\Delta_k^{-it}), U\lambda_k(b)U]
=W_k^*([\Delta_k^{it}a\Delta_k^{-it}, U_kbU_k]\otimes P_\Omega) W_k,
\end{equation*}
where $\Delta_k$ is the modular operator for $\mathbb{G}_k$.
Hence when we see commutators of $\pi({\cal C}_l^k)$ and ${\cal C}_r^l\otimes 1$ in $\mathbb{B}(H\otimes L^2(\mathbb{R}))$, we can first assume $k=l$ since these commutators are zero when $k\neq l$. 
When we see these commutators for a fixed $k$ (and $k=l$), we actually work on $\mathbb{B}(L^2(\mathbb{G}_k)\otimes L^2(\mathbb{R}))$ with the modular action of $\mathbb{G}_k$, where we regard $L^2(\mathbb{G}_k)\simeq\mathbb{C}\Omega\oplus L^2(\mathbb{G}_k)^0\subset H$. 
Hence by the assumption on $\mathbb{G}_k$, we get 
\begin{equation*}
[\pi({\cal C}^k_l), {\cal C}^k_r\otimes 1]\subset \mathbb{K}(L^2(\mathbb{G}_k))\otimes \mathbb{B}(L^2(\mathbb{R})) \subset \mathbb{K}(L^2(\mathbb{G}))\otimes \mathbb{B}(L^2(\mathbb{R})).
\end{equation*}
Thus we get the condition on commutators.
\end{proof}

Now we can give new examples of bi-exact quantum groups and von Neumann algebras with condition $\rm (AOC)^+$.

\begin{Cor}\label{free prod AOC}
Let $\mathbb{G}_i$ $(i=1,\ldots,n)$ be compact quantum groups. Assume that each $\mathbb{G}_i$ is monoidally equivalent to $\mathrm{SU}_q(2)$, $\mathrm{SO}_q(3)$, or $A_u(F)$, where $-1<q<1$, $q\neq 0$, $F$ is not a scalar multiple of a $2$ by $2$ unitary.  Then the free product $\hat{\mathbb{G}}_1*\cdots *\hat{\mathbb{G}}_n$ is bi-exact. 
The associated von Neumann algebra $(L^\infty(\mathbb{G}_1),h_1)\bar{*}\cdots \bar{*}(L^\infty(\mathbb{G}_n),h_n)$ and its Haar state $h_1*\cdots *h_n$ satisfies condition $\rm (AOC)^+$ with the dense $C^*$-algebra $(C_{\rm red}(\mathbb{G}_1),h_1)*\cdots *(C_{\rm red}(\mathbb{G}_n),h_n)$.
\end{Cor}

\subsection{\bf Proof of Theorem \ref{C}}

We first recall some known properties on free quantum groups and quantum automorphism groups. They were originally proved in  $\cite{Ba97}\cite{Wa98_2}\cite{BDV05}$ for free quantum groups and $\cite{RV06}\cite{So08}\cite{Br12}$ for quantum automorphism groups. See $\cite[\textrm{Section 4}]{DFY13}$ for the details.

When $F\in\mathrm{GL}(2,\mathbb{C})$ is a scalar multiple of a $2$ by $2$ unitary, then $A_u(F)=A_u(1_2)$ and the dual of $A_u(1_2)$ is a quantum subgroup of  $\mathbb{Z}*\hat{A}_o(1_2)$.
When $F\in\mathrm{GL}(n,\mathbb{C})$ is any matrix, then the dual of $A_o(F)$ is isomorphic to a free product of some $\hat{A}_o(F_1)$ and $\hat{A}_u(F_1)$ with $F_1\bar{F_1}\in \mathbb{R}\cdot \mathrm{id}$. 
For such a matrix $F$ as $F\bar{F}=c \cdot\mathrm{id}$ for some $c\in\mathbb{R}$, the quantum group $A_o(F)$ is mononidally equivalent to $\mathrm{SU}_q(2)$, where $-\mathrm{Tr}(FF^*)/c=q+q^{-1}$, $-1\leq q\leq 1$ and $q\neq0$. 
When $q=\pm1$, then $\mathrm{dim}_q(u)=|-\mathrm{Tr}(FF^*)/c|=2$, where $u$ is the fundamental representation of $A_o(F)$, and hence $A_o(F)=\mathrm{SU}_{\pm1}(2)$. 
Thus every $\hat{A}_o(F)$ and $\hat{A}_u(F)$ is a discrete quantum subgroup of a free product of amenable discrete quantum groups and duals of compact quantum groups 
which are monoidally equivalent to $\mathrm{SU}_q(2)$ or $A_u(F)$, where $-1< q< 1$, $q\neq0$, $F\in\mathrm{GL}(n,\mathbb{C})$ is not a scalar multiple of a $2$ by $2$ unitary.

The quantum automorphism group $A_{\rm aut}(M,\phi)$ is co-amenable if and only if $\mathrm{dim}(M)\leq4$. 
For any finite dimensional $C^*$-algebra $M$ and any state $\phi$ on $M$, 
$\hat{A}_{\rm aut}(M,\phi)$ is isomorphic to a free product of duals of quantum automorphism groups with $\delta$-form. Such quantum automorphism groups are co-amenable or monoidally equivalent to $\mathrm{SO}_q(3)$, where $\delta=q+q^{-1}$ and $0< q\leq1$. 
When $q=1$ and $\delta=2$, since $\mathrm{dim}(M)\leq\delta^2=4$, $A_{\rm aut}(M,\phi)$ is co-amenable. 
Thus every $\hat{A}_{\rm aut}(M,\phi)$ is a free product of amenable discrete quantum groups and duals of compact quantum groups 
which are monoidally equivalent to $\mathrm{SO}_q(3)$ for some $q$ with $0<q<1$.

We see the following easy lemma before the proof.
\begin{Lem}
Let $\mathbb{G}$ and $\mathbb{H}$ be compact quantum groups. Assume that $\hat{\mathbb{H}}$ is a quantum subgroup of $\hat{\mathbb{G}}$. 
If $\hat{\mathbb{G}}$ is bi-exact (resp.\ $(L^\infty(\mathbb{G}),h)$ satisfies condition $\rm (AOC)^+$ with the $C^*$-algebra $C_{\rm red}(\mathbb{G})$), 
then $\hat{\mathbb{H}}$ is bi-exact (resp.\ $(L^\infty(\mathbb{H}),h)$ satisfies condition $\rm (AOC)^+$ with the $C^*$-algebra $C_{\rm red}(\mathbb{H})$). 
\end{Lem}
\begin{proof}
By assumption there exists the unique Haar state preserving conditional expectation $E_{\mathbb{H}}$ from $L^\infty(\mathbb{G})$ onto $L^\infty(\mathbb{H})$ (and from $C_{\rm red}(\mathbb{G})$ onto $C_{\rm red}(\mathbb{H})$). 
It extends to a projection $e_{\mathbb{H}}$ from $L^2(\mathbb{G})$ onto $L^2(\mathbb{H})$. 
Let $\theta$ be a u.c.p.\ map as in the definition of bi-exactness (resp.\ condition $\rm (AOC)^+$). Then a u.c.p.\ map given by $a\otimes b^{\rm op}\mapsto e_{\mathbb{H}}\theta(a\otimes b^{\rm op})e_{\mathbb{H}}$ for $a, b\in C_{\rm red}(\mathbb{H})$ 
(resp.\ $a\otimes b^{\rm op}\mapsto (e_{\mathbb{H}}\otimes 1)\theta(a\otimes b^{\rm op})(e_{\mathbb{H}}\otimes 1)$ for $a, b\in C_{\rm red}(\mathbb{H})\rtimes_{\rm r}\mathbb{R}$) do the work. 
Note that modular objects $J$ and $\Delta^{it}$ of the Haar state commute with $e_{\mathbb{H}}$. 
Local reflexivity of $C_{\rm red}(\mathbb{H})$ (resp.\ $C_{\rm red}(\mathbb{H})\rtimes_{\rm r}\mathbb{R}$) follows from that of $C_{\rm red}(\mathbb{G})$ (resp.\ $C_{\rm red}(\mathbb{G})\rtimes_{\rm r}\mathbb{R}$) since it is a subalgebra.
\end{proof}

\begin{proof}[\bf Proof of Theorem \ref{C}]
For weak amenablity and the $\rm W^*$CBAP, they are already discussed in $\cite[\textrm{Section 5}]{DFY13}$ (see also $\cite[\textrm{Theorem 4.6}]{Fr11}$). Hence we see only bi-exactness and condition $\rm (AOC)^+$.

Let $\mathbb{G}$ be as in the statement. Then thanks for the observation above, $\hat{\mathbb{G}}$ is a discrete quantum subgroup of $\hat{\mathbb{G}}_1*\cdots *\hat{\mathbb{G}}_n$, where each $\mathbb{G}_i$ is 
co-amenable, a dual of bi-exact discrete group, or monoidally equivalent to $\mathrm{SU}_q(2)$, $\mathrm{SO}_q(3)$, or $A_u(F)$, where $-1<q<1$, $q\neq 0$, $F$ is not a scalar multiple of a $2$ by $2$ unitary.
Hence by Theorems \ref{example st bi-exact} and \ref{example bi-exact}, Remarks \ref{amenable} and \ref{group}, Propositions \ref{free prod bi-exact} and \ref{free prod AOC^+}, and the last lemma, $\hat{\mathbb{G}}$ is bi-exact and $(L^\infty(\mathbb{G}),h)$ satisfies condition $\rm (AOC)^+$ with $C_{\rm red}(\mathbb{G})$.
\end{proof}

\section{\bf Proofs of main theorems}

\subsection{\bf Proof of Theorem \ref{A}}

To prove Theorem \ref{A}, we can use the same manner as that in $\cite[\rm Theorem\ 3.1]{PV12}$ except for (i) Proposition 3.2 and (ii) Subsection 3.5 (case 2) in $\cite{PV12}$.
 
To see (ii) in our situation, we need a structure of quantum group von Neumann algebras, which is weaker than that of group von Neumann algebras but enough to solve our problem. Since we will see a similar (and more general) phenomena in the next subsection (Lemma \ref{case2}), we omit it.

Hence here we give a proof of (i). To do so, we see one proposition which is a quantum analogue of a well known property for crossed products with discrete groups.
\begin{Pro}
Let $\mathbb{G}$ be a compact quantum group of Kac type, whose dual acts on a tracial von Neumann algebra $(B,\tau_B)$ as a trace preserving action. Write $M:= \hat{\mathbb{G}}\ltimes B$. Let $p$ be a projection in $M$ and $A\subset pMp$ a von Neumann subalgebra. Then the following conditions are equivalent:

\begin{itemize}
	\item[$\rm (i)$] $A\not\preceq_M B$;
	\item[$\rm (ii)$] there exists a net $(w_n)_n$ of unitaries in $A$ such that  $\lim_n \|(w_n)_{i,j}^x \|_{2,\tau_B}=0$ for any $x\in \mathrm{Irred}(\mathbb{G})$ and $i,j$, where $(w_n)_{i,j}^x$ is given by $w_i=\sum_{x\in \mathrm{Irred}(\mathbb{G}), i,j}(w_n)_{i,j}^x u_{i,j}^x$ for a fixed basis $(u_{i,j}^x)$ satisfying $h(u_{i,j}^x u_{k,l}^{y*})=\delta_{i,k}\delta_{j,l}\delta_{x,y}h(u_{i,j}^x u_{i,j}^{x*})$.
\end{itemize}
\end{Pro}
\begin{proof}
We first assume (i). Then by definition, there exists a net $(w_n)_n$ in ${\cal U}(A)$ such that  $\lim_n \|E_B(b^*w_na) \|_2=0$ for any $a,b \in M$. Putting $b=1$ and $a=u_{k,l}^{y*}$ and since $E_B(w_n u_{k,l}^{y*})=\sum_{x,i,j}(w_n)_{i,j}^x h(u_{i,j}^x u_{k,l}^{y*})=(w_n)_{k,l}^y h(u_{k,l}^y u_{k,l}^{y*})$, we have 
\begin{equation*}
\|(w_n)_{k,l}^y\|_2= |h(u_{k,l}^y u_{k,l}^{y*})|^{-1}\|E_B(w_n u_{k,l}^{y*})\|_2\rightarrow 0 \quad (n\rightarrow \infty).
\end{equation*}

Conversely, assume (ii). We will show $\lim_n \|E_B(b^*w_na) \|_2=0$ for any $a,b \in M$. To see this, we may assume $a=u_{\alpha,\beta}^y$ and $b=u_{k,l}^z$.
For any $c\in B$, $u_{k,l}^{z*}c$ is a linear combination of $c' u_{k',l'}^{z*}$ for some $c'\in B$ and $k',l'$. When we apply $E_B$ to $u_{k,l}^{z*} c u_{i,j}^x u_{\alpha,\beta}^y$, it does not vanish only if $\bar{z}\otimes x\otimes y$ contains the trivial representation. For fixed $y,z$, the number of such $x$ is finite (since this means $x\in z\otimes \bar{y}$). So we have 

\begin{eqnarray*}
\|E_B(b^*w_na) \|_2 &=& \|\sum_{\textrm{finitely many }x, i,j} E_B(b^*(w_n)_{i,j}^x u_{i,j}^xa) \|_2\\
&\leq& \sum_{\textrm{finitely many }x, i,j} \|b^*(w_n)_{i,j}^x u_{i,j}^xa \|_2\\
&\leq& \sum_{\textrm{finitely many }x, i,j} \|b^*\|\|(w_n)_{i,j}^x \|_2 \|u_{i,j}^xa\| \rightarrow 0.
\end{eqnarray*}
\end{proof}

The following proposition is a corresponding one to (i) Proposition 3.2 in $\cite{PV12}$.
\begin{Pro}
Theorem $\ref{A}$ is true if it is true for any trivial action.
\end{Pro}
\begin{proof}
Let $\mathbb{G}$, $B$, $M$, $p$, and $A$ be as in Theorem \ref{A}.
By Fell's absorption principle, we have the following $*$-homomorphism:
\begin{equation*}
\Delta\colon M= \hat{\mathbb{G}}\ltimes B \ni bu_{i,j}^x \longmapsto \sum_{k=1}^{n_x} bu_{i,k}^x\otimes u_{k,j}^x\in M\otimes L^\infty(\mathbb{G})=:{\cal M}.
\end{equation*}
Put ${\cal A}:=\Delta(A)$, $\tilde{q}:=\Delta(q)$ and ${\cal P}:={\cal N}_{\tilde{q}{\cal M}\tilde{q}}({\cal A})''$. Then consider following statements:

\begin{itemize}
	\item[(1)] If $A$ is amenable relative to $B$ in $M$, then $\cal A$ is amenable relative to $M\otimes1$ in $\cal M$.
	\item[(2)] If $\cal P$ is amenable relative to $M\otimes1$ in $\cal M$, then ${\cal N}_{qMq}(A)''$ is amenable relative to $B$ in $M$.
	\item[(3)] If ${\cal A}\preceq_{\cal M}M\otimes 1$, then $A\preceq_M B$.
\end{itemize}
If one knows them, then (1) and our assumption say that we have either $\cal P$ is amenable relative to $M\otimes1$ or ${\cal A}\preceq_{\cal M}M\otimes 1$. Then (2) and (3) imply our desired conclusion.

To show (1) and (2), we only need the property $\Delta\circ E_B=E_{M\otimes 1}\circ \Delta=E_{B\otimes1}\circ \Delta$. So we can use the same strategy as in the group case. To show (3) we need the previous proposition, and once we accept it, then (3) is easily verified. 
\end{proof}

\subsection{\bf Proof of Theorem \ref{B}}

We use the same symbol as in the statement of Theorem \ref{B}. We moreover use the following notation:
\begin{eqnarray*}
&&P:= {\cal N}_{qMq}(A)'', \ 
C_h:=C_h(L^\infty(\mathbb{G})),\ 
\tau:=\mathrm{Tr}_M(q\cdot q),\\
&&L^2(M):=L^2(M,\mathrm{Tr}_M)=L^2(B,\tau_B)\otimes L^2(C_h,\mathrm{Tr}),\\
&&\pi\colon M=B\otimes C_h \ni b\otimes x \mapsto (b\otimes_A q \otimes x) \in \mathbb{B}((L^2(M)q \otimes_A L^2(P))\otimes L^2(C_h,\mathrm{Tr})),\\
&&\theta\colon P^{\rm op} \ni y^{\rm op}\mapsto (q\otimes_A y^{\rm op}) \otimes 1 \in \mathbb{B}((L^2(M)q \otimes_A L^2(P))\otimes L^2(C_h,\mathrm{Tr})),\\
&&N:=W^*\{\pi(B\otimes 1), \theta(P^{\rm op})\},\ {\cal N}:=N\otimes C_h.
\end{eqnarray*}

We first recall the following theorem which is a generalization of $\cite[\textrm{Theorem 3.5}]{OP07}$ and $\cite[\textrm{Theorem B}]{Oz10}$. 
Now we are working on a semifinite von Neumann algebra $\cal N$ but still the theorem is true. 
To verify this, we need the following observation. 
If a semifiite von Neumann algebra has the $\rm W^*$CBAP, $(\phi_i)_i$ is an approximation identity, and $(p_j)_j$ is a net of trace-finite projections converging to 1 strongly, then a subnet of $(\psi_j\circ\phi_i)$, where $\psi_j$ is a compression by $p_j$, is an approximation identity. 
Hence we can find an approximate identity whose images are contained in a trace-finite part of the semifinite von Neumann algebra. 
As finite rank maps relative to $B$ $\cite[\textrm{Definition 5.3}]{PV11}$, one can take linear spans of 
\begin{equation*}
\psi_{y,z,r,t}\colon qMq\rightarrow qMq; x\mapsto y(\mathrm{id}_B\otimes \mathrm{Tr})(zxr)t,
\end{equation*}
where $y,z,r,t\in M$ satisfying $y=qy$, $t=tq$, $z=(1\otimes p)z$, and $r=r(1\otimes p')$ for some Tr-finite projections $p,p'\in C_h$. Note that for fixed $p\in C_h$ with $\mathrm{Tr}(p)<\infty$, $\mathrm{Tr}(p)^{-1}(\mathrm{id}_B\otimes \mathrm{Tr})$ is a conditional expectation from $B\otimes pC_hp$ onto $B$.

\begin{Thm}[{\cite[\textrm{Theorem 5.1}]{PV12}}]
There exists a net $(\omega_i)_i$ of normal states on $\pi(q){\cal N}\pi(q)$ such that  
\begin{itemize}
	\item[$\rm (i)$] $\omega_i(\pi(x))\rightarrow \tau(x)$ for any $x\in qMq$;
	\item[$\rm (ii)$] $\omega_i(\pi(a)\theta(\bar{a}))\rightarrow 1$ for any $a\in {\cal U}(A)$;
	\item[$\rm (iii)$] $\|\omega_i\circ\mathrm{Ad}(\pi(u)\theta(\bar{u}))-\omega_i\| \rightarrow 0$ for any $u \in {\cal N}_{qMq}(A)$. 
\end{itemize}
Here $\bar a$ means $(a^{\rm op})^*$.
\end{Thm}

Let $H$ be a standard Hilbert space of $N$ with a canonical conjugation $J_N$. 
From now on, as a standard representation of $C_h$, we use $L^2(C_h):=L^2(C_h,\hat{h})=L^2(\mathbb{G})\otimes L^2(\mathbb{R})$, where $\hat{h}$ is the dual weight of $h$. 
Then since $H\otimes L^2(C_h)$ is standard for $\cal N$ with a canonical conjugation ${\cal J}:=J_N\otimes J_{C_h}$, there exists a unique net $(\xi_i)_i$ of unit vectors in the positive cone of $\pi(q){\cal J}\pi(q){\cal J}(H\otimes L^2(C_h))$ such that $\omega_i(\pi(q)x\pi(q))=\langle x \xi_i | \xi_i \rangle$ for  $x\in {\cal N}$. Then  conditions on $\omega_i$ are translated to the following conditions: 
\begin{itemize}
	\item[$\rm (i)'$] $\langle \pi(x) \xi_i | \xi_i \rangle\rightarrow \tau(qxq)$ for any $x\in M$;
	\item[$\rm (ii)'$] $\|\pi(a)\theta(\bar a)\xi_i -\xi_i\|\rightarrow 0$ for any $a\in {\cal U}(A)$;
	\item[$\rm (iii)'$] $\|\mathrm{Ad}(\pi(u)\theta(\bar{u}))\xi_i -\xi_i\| \rightarrow 0$ for any $u \in {\cal N}_{qMq}(A)$. 
\end{itemize}
Here $\mathrm{Ad}(x)\xi_i:=x{\cal J}x{\cal J}\xi_i$. To see $\rm (iii)'$, we need a generalized Powers--St$\rm \o$rmer inequality (e.g.\ $\cite[\textrm{Theorem IX.1.2.(iii)}]{Ta2}$).

The following lemma is a very similar statement to that in $\cite[\textrm{Subsection 3.5 (case 2)}]{PV12}$. Since we treat quantum groups and moreover our object $C_h$ is semifinite and twisted (not a tensor product), we need more careful treatment. So here we give a complete proof of the lemma.
\begin{Lem}\label{case2}
Let $(\xi_i)_i$ be as above. Assume $A\not\preceq_{M} B\otimes L\mathbb{R}$.
Then we have
\begin{equation*}
\limsup_i\| (1_{\mathbb{B}(H)}\otimes x \otimes 1_{L\mathbb{R}})\xi_i\|=0 \quad \textrm{for any } x \in \mathbb{K}(L^2(\mathbb{G})).
\end{equation*}
\end{Lem}
\begin{proof}
Suppose by contradiction that there exist $\delta>0$ and a finite subset ${\cal F}\subset \mathrm{Irred}(\mathbb{G})$ such that 
\begin{equation*}
\limsup_i\| (1_{\mathbb{B}(H)}\otimes P_{\cal F} \otimes 1_{L\mathbb{R}})\xi_i\| >\delta,
\end{equation*}
where $P_{\cal F}$ is the orthogonal projection onto $\sum_{x\in {\cal F}}H_x \otimes H_{\bar x}$.
Replacing with a subnet of $(\xi_i)_i$, we may assume that 
\begin{equation*}
\liminf_i\| (1_{\mathbb{B}(H)}\otimes P_{\cal F} \otimes 1_{L\mathbb{R}})\xi_i\| >\delta.
\end{equation*}
Our goal is to find a finite set ${\cal F}_1\subset \mathrm{Irred}(\mathbb{G})$ and a subnet of $(\xi_i)_i$ which satisfy  
\begin{equation*}
\liminf_i\| (1_{\mathbb{B}(H)}\otimes P_{{\cal F}_1} \otimes 1_{L\mathbb{R}})\xi_i\| >2^{1/2}\delta.
\end{equation*}
Then repeating this argument, we get a contradiction since 
\begin{equation*}
1\geq\liminf_i\| (1_{\mathbb{B}(H)}\otimes P_{{\cal F}_k} \otimes 1_{L\mathbb{R}})\xi_i\| >2^{k/2}\delta \quad (k\in \mathbb{N}).
\end{equation*}

\noindent
\textit{{\bf claim 1.} There exists a $\mathrm{Tr}$-finite projection $r\in L\mathbb{R}$ and a subnet of $(\xi_i)_i$ such that  }
\begin{equation*}
\liminf_i\| (1_{\mathbb{B}(H)}\otimes P_{\cal F} \otimes r)\xi_i\| >\delta.
\end{equation*}
\begin{proof}[\bf proof of claim 1]
Define a state on $\mathbb{B}(H\otimes L^2(C_h))$ by $\Omega(X):=\mathrm{Lim}_i\langle X\xi_i | \xi_i\rangle$. Then the condition $\rm (i)'$ of $(\xi_i)_i$ says that $\Omega(\pi(x))=\tau(qxq)$ for $x\in M$. Let $r_j$ be a net of $\mathrm{Tr}$-finite projections in $L\mathbb{R}$ which converges to 1 strongly. Then we have 
\begin{eqnarray*}
|\Omega((1\otimes P_{\cal F}\otimes 1) \pi(1-r_j))|^2&\leq&\Omega((1\otimes P_{\cal F}\otimes 1))\Omega(\pi(1-r_j)) \\
&=&\Omega((1\otimes P_{\cal F}\otimes 1))\tau(q(1-r_j)q) \rightarrow 0 \quad (j\rightarrow \infty).
\end{eqnarray*}
This implies that
\begin{eqnarray*}
0&\leq&
\mathrm{Lim}_i\| (1\otimes P_{\cal F} \otimes 1)\xi_i\| - \mathrm{Lim}_i\| (1\otimes P_{\cal F} \otimes r_j)\xi_i\| \\
&=&\Omega(1\otimes P_{\cal F} \otimes 1) - \Omega( (1\otimes P_{\cal F} \otimes 1)\pi(r_j))\\
&=&\Omega((1\otimes P_{\cal F}\otimes 1) \pi(1-r_j)) \rightarrow 0 \quad (j\rightarrow \infty)
\end{eqnarray*}
Hence we can find a $\mathrm{Tr}$-finite projection $r\in L\mathbb{R}$ such that  $\mathrm{Lim}_i\| (1\otimes P_{\cal F} \otimes r)\xi_i\|>\delta$. Finally by taking a subnet of $(\xi_i)_i$, we have $\liminf_i\| (1\otimes P_{\cal F} \otimes r)\xi_i\|>\delta$.
\end{proof}

We fix a basis $\{u_{i,j}^x \}_{i,j}^x (=:X)$ of the dense Hopf $*$-algebra of $C(\mathbb{G})$ and use the notation 
\begin{eqnarray*}
&&\mathrm{Irred}(\mathbb{G})_x:=\{ u_{i,j}^x\in  X \mid i,j=1,\ldots,n_x\} \quad (x\in \mathrm{Irred}(\mathbb{G})),\\
&&\mathrm{Irred}(\mathbb{G})_{\cal E}:=\cup_{x\in{}\cal E}\mathrm{Irred}(\mathbb{G})_x \quad ({\cal E} \subset \mathrm{Irred}(\mathbb{G})).
\end{eqnarray*}
We may assume that each $\hat{u}_{i,j}^x\in L^2(\mathbb{G})$ is an eigenvector of the modular operator of $h$, namely, for any $u_{i,j}^x$ there exists $\lambda>0$ such that  $\Delta_h^{it}\hat{u}_{i,j}^x=\lambda^{it}\hat{u}_{i,j}^x$ $(t\in \mathbb{R})$. 
In this case we have a formula $\sigma_t^h(u_{i,j}^x)=\lambda^{it}u_{i,j}^x$ and hence $\sigma_t^h(u_{i,j}^{x*} a u_{i,j}^{x})=u_{i,j}^{x*} \sigma_t^h(a) u_{i,j}^{x}$ for any $a\in L^\infty(\mathbb{G})$.

Let $P_e$ be the orthogonal projection from $L^2(\mathbb{G})$ onto $\mathbb{C}\hat 1$.
For any $u\in X$, consider a compression map 
\begin{equation*}
\Phi_u(x):=h(u^*u)^{-1}(1\otimes P_eu^* \otimes 1)\pi(x)(1\otimes uP_e \otimes 1)
\end{equation*}
for $x\in M$, which gives a normal map from $M$ into $B\otimes \mathbb{C}P_e\otimes\mathbb{B}(L^2(\mathbb{R}))\simeq B\otimes \mathbb{B}(L^2(\mathbb{R}))$.

\noindent
\textit{{\bf claim 2.} 
For any $u\in X$, we have}
\begin{equation*}
\Phi_{u}(b\otimes af)= b\otimes h(u^*u)^{-1}h(u^{*}au)f \quad (b\in B, a\in L^\infty(\mathbb{G}), f\in L\mathbb{R}).
\end{equation*}
\textit{In particular, $\Phi_u$ is a normal conditional expectation from $M$ onto $B\otimes L\mathbb{R}$.}
\begin{proof}[\bf proof of claim 2]
Assume $B=\mathbb{C}$ for simplicity. Recall that any element $a\in L^\infty(\mathbb{G})$ in the continuous core is of the form $a=\int\sigma_{-t}^h(a)\otimes e_t \cdot dt$, which means $(a\xi)(s)=\sigma_{-s}^h(a)\xi(s)$ for $\xi \in L^2(\mathbb{G})\otimes L^2(\mathbb{R})$ and $s\in \mathbb{R}$. 
Hence a simple calculation shows that for any $a\in L^\infty(\mathbb{G})$,
\begin{eqnarray*}
h(u^*u)\Phi_u(a)&=&(P_eu^* \otimes 1)\int\sigma_{-t}^h(a)\otimes e_t \cdot dt (uP_e \otimes 1)\\
&=&\int P_eu^*\sigma_{-t}^h(a)uP_e\otimes e_t \cdot dt\\
&=&\int h(u^*au)P_e\otimes e_t \cdot dt= h(u^*au)P_e\otimes 1,
\end{eqnarray*}
where we used $P_eu^*\sigma_{-t}^h(a)uP_e=h(u^*\sigma_{-t}^h(a)u)P_e=h(\sigma_{-t}^h(u^*au))P_e=h(u^*au)P_e$. Thus $\Phi_u$ satisfies our desired condition.
\end{proof}

\noindent
\textit{{\bf claim 3.} For any $u\in X$ and $x\in M$, we have}
\begin{equation*}
\limsup_i\|\pi(x)(1\otimes P_u\otimes r)\xi_i\|\leq h(u^*u)^{-1/2} \|x(1\otimes r)\|_{2,\mathrm{Tr}_M\circ\Phi_u},
\end{equation*}
\textit{where $P_u$ is the orthogonal projection from $L^2(\mathbb{G})$ onto $\mathbb{C}\hat{u}$.}
\begin{proof}[\bf proof of claim 3]
This follows from a direct calculation. Since $P_u=h(u^*u)^{-1} uP_e u^*$, we have
\begin{eqnarray*}
&&h(u^*u)\|(\pi(x)(1\otimes P_u\otimes r)\xi_i\|^2\\
&=&h(u^*u)\langle (1\otimes P_u \otimes 1)\pi(\tilde{x}^*\tilde{x}) (1\otimes P_u \otimes 1)\xi_i | \xi_i\rangle \qquad (\tilde x := x(1\otimes r))\\
&=&h(u^*u)^{-1}\langle (1\otimes P_eu^* \otimes 1)\pi(\tilde{x}^*\tilde{x}) (1\otimes uP_e \otimes 1)\tilde{\xi}_i | \tilde{\xi}_i\rangle \qquad (\tilde{\xi}_i :=(1\otimes u^* \otimes 1) \xi_i)\\
&=&\langle (1\otimes P_e \otimes 1)\pi(\Phi_u(\tilde{x}^*\tilde{x})) (1\otimes P_e \otimes 1)\tilde{\xi}_i | \tilde{\xi}_i\rangle\\
&=&\| \pi(y) (1\otimes P_eu^* \otimes 1)\xi_i \|^2 \qquad (y:=\Phi_u(\tilde{x}^*\tilde{x})^{1/2})\\
&\leq& \| \pi(y) \xi_i \|^2 \qquad (\textrm{since $\pi(y)$ and $(1\otimes P_eu^* \otimes 1$) commute})\\
&\rightarrow& \mathrm{Tr}_M(qy^*yq) \leq \mathrm{Tr}_M(y^*y) = \mathrm{Tr}_M(\Phi_u(\tilde{x}^*\tilde{x}))= \mathrm{Tr}_M\circ\Phi_u((1\otimes r)x^*x(1\otimes r)).
\end{eqnarray*}
\end{proof}

\noindent
\textit{{\bf claim 4.} For any $\epsilon>0$ and any finite subset ${\cal E}\subset \mathrm{Irred}(\mathbb{G})$, there exist $a\in {\cal U}(A)$ and $v\in M$ such that  
\begin{itemize}
	\item $v$ is a finite sum of elements of the form $b\otimes u f$, where $b\in B$, $f\in L\mathbb{R}$ and $u\in X \setminus \mathrm{Irred}(\mathrm{G})_{\cal E}$;
	\item $h(u^*u)^{-1/2}\|(a-v)(1\otimes r)\|_{2,\mathrm{Tr}_M\circ\Phi_u}<\epsilon$ for any $u\in\mathrm{Irred}(\mathrm{G})_{\cal F}$.
\end{itemize}
}
\begin{proof}[\bf proof of claim 4]
Since $A\not\preceq_{{}_e M_e}B\otimes L\mathbb{R}r$, where $e:=q\vee r$, for any $\epsilon>0$ and any finite subset ${\cal E}\subset \mathrm{Irred}(\mathbb{G})$, there exist $a\in {\cal U}(A)$ s.t\ $\|E_{B\otimes L\mathbb{R}}((1\otimes r)(1\otimes u^*)a(1\otimes r))\|_{2,\mathrm{Tr}_M}<\epsilon$ for any $u\in \mathrm{Irred}(\mathrm{G})_{\cal E}$.

Since the linear span of $\{b\otimes u\lambda_t| b\in B, t\in\mathbb{R},u\in X\}$ is strongly dense in $M$, we can find a bounded net $(z_j)$ of elements in this linear span which converges to $a$ in the strong topology. 
Hence we can find $z$ in the linear span such that  $\|(a-z)(1\otimes r)\|_{2,\mathrm{Tr}_M}$ and $\|(a-z)(1\otimes r)\|_{2,\mathrm{Tr}_M\circ\Phi_u}$ $(u\in\mathrm{Irred}(\mathrm{G})_{\cal F})$ are very small. 
In this case we may assume that $\|E_{B\otimes L\mathbb{R}}((1\otimes r)(1\otimes u^*)z(1\otimes r))\|_{2,\mathrm{Tr}_M}<\epsilon$ for any $u\in \mathrm{Irred}(\mathrm{G})_{\cal E}$.

Write $z=\sum_{\rm finite} b_{i,j}^x\otimes u_{i,j}^x f_{i,j}^x$. Then for any $y\in{\cal E}$ the above inequality implies 
\begin{eqnarray*}
\epsilon&>&\|E_{B\otimes L\mathbb{R}}((1\otimes r)(1\otimes u^{y*}_{k,l})z(1\otimes r))\|_{2,\mathrm{Tr}_M}\\
&=&\|\sum_{\rm finite} (1\otimes r)(b_{i,j}^x\otimes h(u^{y*}_{k,l}u_{i,j}^x)f_{i,j}^x)(1\otimes r)\|_{2,\mathrm{Tr}_M}\\
&=&\sum_{\rm finite}\delta_{x,y}\delta_{i,k}\delta_{j,l}h(u^{y*}_{k,l}u_{i,j}^x)\|b_{i,j}^x\otimes f_{i,j}^xr\|_{2,\mathrm{Tr}_M}
=\|u^{y}_{k,l}\|^2_{2,h}\|b_{k,l}^y\otimes f_{k,l}^yr\|_{2,\mathrm{Tr}_M}.
\end{eqnarray*}
Hence if we write $z=\sum_{x\in {\cal E}}b_{i,j}^x\otimes u_{i,j}^x f_{i,j}^x + \sum_{x\not\in {\cal E}}b_{i,j}^x\otimes u_{i,j}^x f_{i,j}^x $, and say $v$ for the second sum, then we have 
\begin{eqnarray*}
\|(z-v)(1\otimes r)\|_{2,\mathrm{Tr}_M\circ\Phi_u}
&\leq&\sum_{x\in{\cal E},i,j}\| b_{i,j}^x\otimes u_{i,j}^x f_{i,j}^xr \|_{2,\mathrm{Tr}_M\circ\Phi_u} \\
&\leq&\sum_{x\in{\cal E},i,j}\|u_{i,j}^x\| \| b_{i,j}^x\otimes f_{i,j}^xr \|_{2,\mathrm{Tr}_M\circ\Phi_u} \\
&\leq&\sum_{x\in{\cal E},i,j}\| b_{i,j}^x\otimes f_{i,j}^xr \|_{2,\mathrm{Tr}_M} \\
&<&C({\cal E})\cdot\epsilon,
\end{eqnarray*}
for any $u\in\mathrm{Irred}(\mathrm{G})_{\cal F}$, where $C({\cal E})>0$ is a constant which depends only on ${\cal E}$. Thus this $v$ is our desired one.
\end{proof}

Now we return to the proof. Since $1\otimes P_{\cal F}\otimes r$ is a projection, we have
\begin{eqnarray*}
\limsup_i\|\xi_i-(1\otimes P_{\cal F}\otimes r)\xi_i\|^2
&=&\limsup_i(\|\xi_i\|^2-\|(1\otimes P_{\cal F}\otimes r)\xi_i\|^2)\\
&=&1-\liminf_i\|(1\otimes P_{\cal F}\otimes r)\xi_i\|^2<1-\delta^2.
\end{eqnarray*}
So there exists $\epsilon>0$ such that  $\limsup_i\|\xi_i-(1\otimes P_{\cal F}\otimes r)\xi_i\|<(1-\delta^2)^{1/2}-\epsilon$. We apply claim 4 to $(\sum_{x\in{\cal F}}\mathrm{dim}(H_x)^2)^{-1}\epsilon$ and ${\cal E}:={\cal F}\bar{\cal F}$ $(=\{z| z\in x\otimes \bar{y} \textrm{ for some }x,y\in{\cal F}\})$, and get $a$ and $v$. Then we have
\begin{eqnarray*}
&&\limsup_i\|\xi_i-\theta(\bar{a})\pi(v)(1\otimes P_{\cal F}\otimes r)\xi_i\|\\
&\leq&\limsup_i\|\xi_i-\theta(\bar{a})\pi(a)(1\otimes P_{\cal F}\otimes r)\xi_i\|+\limsup_i\|\theta(\bar{a})\pi(a-v)(1\otimes P_{\cal F}\otimes r)\xi_i\|\\
&\leq&\limsup_i\|\theta(a^{\rm op})\pi(a^*)\xi_i-(1\otimes P_{\cal F}\otimes r)\xi_i\| + \sum_{u\in \mathrm{Irred}(\mathbb{G})_{\cal F}} h(u^*u)^{-1/2}\|(a-v)(1\otimes r)\|_{2, \mathrm{Tr}_M\circ\Phi_u} \\
&<&\limsup_i\|\xi_i-(1\otimes P_{\cal F}\otimes r)\xi_i\| + \epsilon < (1-\delta^2)^{1/2},
\end{eqnarray*}
where we used condition $\rm (ii)'$ of $(\xi_i)$ and claim 3. 
By the choice of $v$, it is of the form $\sum_{x\in{\cal S},i,j}b_{i,j}^x\otimes u_{i,j}^x f_{i,j}^x$ for some finite set ${\cal S}\subset \mathrm{Irred}(\mathbb{G})$ with ${\cal S}\cap{\cal F}\bar{\cal F}=\emptyset$. Note that this means ${\cal SF}\cap{\cal F}=\emptyset$.
The vector $\theta(\bar{a})\pi(v)(1\otimes P_{\cal F}\otimes r)\xi_i$ is contained in the range of $1\otimes P_{\cal SF}\otimes 1$. This is because the modular operator $\Delta_h^{it}$ commutes with $P_{\cal SF}$ and $P_{\cal F}$ and hence 
\begin{equation*}
(1\otimes P_{{\cal SF}}\otimes 1)\pi(v)(1\otimes P_{\cal F}\otimes 1)=\pi(v)(1\otimes P_{\cal F}\otimes 1).
\end{equation*} 
Then we have
\begin{eqnarray*}
1-\delta^2
&>&\limsup_i\|\xi_i-\theta(\bar{a})\pi(v)(1\otimes P_{\cal F}\otimes r)\xi_i\|^2 \\
&\geq&\limsup_i\|(1\otimes P_{\cal SF}\otimes 1)^{\perp}\xi_i\|^2\\
&=&1-\liminf_i\|(1\otimes P_{\cal SF}\otimes 1)\xi_i\|^2.
\end{eqnarray*}
This means $\delta<\liminf_i\|(1\otimes P_{\cal SF}\otimes 1)\xi_i\|$. 
Finally put ${\cal F}_1:={\cal F}\cup{\cal SF}$. Then since ${\cal SF}\cap{\cal F}=\emptyset$, we get
\begin{equation*}
\liminf_i\|(1\otimes P_{{\cal F}_1}\otimes 1)\xi_i\|^2
\geq\liminf_i\|(1\otimes P_{\cal F}\otimes 1)\xi_i\|^2 + \liminf_i\|(1\otimes P_{\cal SF}\otimes 1)\xi_i\|^2>2\delta^2.
\end{equation*}
Thus ${\cal F}_1$ is our desired one and we can end the proof.
\end{proof}

Now we start the proof. We follow that of $\cite[\textrm{Subsection 3.4 (case 1)}]{PV12}$. Since this part of the proof does not rely on the structure of group von Neumann algebras, we need only slight modifications, which were already observed in $\cite{Is12_2}$. Hence here we give a rough sketch of the proof.
We use the following notation which is used in $\cite{PV12}$.
\begin{eqnarray*}
&&{\cal D}:=M\odot M^{\rm op} \odot P^{\rm op} \odot P \supset (C_{\rm red}(\mathbb{G})\rtimes_{\rm r}\mathbb{R})\odot (C_{\rm red}(\mathbb{G})\rtimes_{\rm r}\mathbb{R})^{\rm op} \odot P^{\rm op} \odot P =:{\cal D}_0,\\
&&\Psi\colon {\cal D}\rightarrow \mathbb{B}(H\otimes L^2(C_h)\otimes L^2(C_h)), \quad 
\Theta \colon {\cal D}\rightarrow \mathbb{B}(H\otimes L^2(C_h)),\\
&&\Psi((b_1\otimes x_1) \otimes (b_2\otimes x_2)^{\rm op}\otimes y^{\rm op}\otimes z)=b_1J_Nb_2^*J_Ny^{\rm op}J\bar{z}J\otimes x_1\otimes x_2^{\rm op}, \\
&&\Theta((b_1\otimes x_1) \otimes (b_2\otimes x_2)^{\rm op}\otimes y^{\rm op}\otimes z)=b_1J_Nb_2^*J_Ny^{\rm op}J\bar{z}J\otimes x_1x_2^{\rm op}\\
&&\phantom{\Theta((b_1\otimes x_1) \otimes (b_2\otimes x_2)^{\rm op}\otimes y^{\rm op}\otimes z)}=\pi(b_1\otimes x_1){\cal J}\pi(b_2\otimes x_2)^*{\cal J}\theta(y^{\rm op}){\cal J}\theta(\bar{z}){\cal J}.
\end{eqnarray*}

\begin{proof}[\bf Proof of Theorem \ref{B}]
Define a state on $\mathbb{B}(H\otimes L^2(C_h))$ by $\Omega_1(X):=\mathrm{Lim}_i\langle X\xi_i|\xi_i\rangle$. Our condition $\rm (AOC)^+$, together with Lemma \ref{case2}, implies that $|\Omega_1(\Theta(X))|\leq \|\Psi(X)\|$ for any $X\in{\cal D}_0$. 
We extend this inequality on $\cal D$ by using an approximation identity of $C_h$, which take vales in $C_{\rm red}(\mathbb{G})\rtimes_{\rm r}\mathbb{R}$ (such a net exists, see $\cite[\textrm{Lemma 2.3.1}]{Is12_2}$).
Then a new functional $\Omega_2$ on $C^*(\Psi({\cal D}))$ is defined by $\Omega_2(\Psi(X)):=\Omega_1(\Theta(X))$ $(X\in{\cal D})$.
 Its Hahn--Banach extension on $\mathbb{B}(H\otimes L^2(C_h)\otimes L^2(C_h))$ restricts to a $P$-central state on $q(B\otimes \mathbb{B}(C_h))q\otimes \mathbb{C}$.  More precisely we have a $P$-central state on $\mathbb{B}(qL^2(M))\cap(B^{\rm op})'(=q(B\otimes \mathbb{B}(C_h))q)$ which restricts to the trace $\tau$ on $qMq$.
\end{proof}

\subsection{\bf Proofs of corollaries}
For the Kac type case, the same proofs as in the group case work. So we see only the proof of Corollary \ref{B}. 

Let $\mathbb{G}$ be a compact quantum group in the statement of the corollary, $h$ the Haar state of $\mathbb{G}$, and $(B,\tau_B)$ be a tracial von Neumann algebra. 
Write $M:=B\otimes L^\infty(\mathbb{G})$. Let $(A,\tau_A)$ be an amenable tracial von Neumann subalgebra in $M$ with expectation $E_A$. 
Then since modular actions of $\tau_A$ and $\tau_A\circ E_A$ has the relation $\sigma^{\tau_A}=\sigma^{\tau_A\circ E_A}|_A$ (see the proof of $\cite[\rm{Theorem\ IX.4.2}]{Ta2}$), 
we have an inclusion $A\otimes L\mathbb{R}=C_{\tau_A}(A)\subset C_{\tau_A\circ E_A}(M)$ with a faithful normal conditional expectation $\tilde{E}_A$ given by $\tilde{E}_A(x\lambda_t)=E_A(x)\lambda_t$ for $x\in M$ and $t\in \mathbb{R}$. 
Since continuous cores are canonically isomorphic with each other, there is a canonical isomorphism from $C_{\tau_A\circ E_A}(M)$ onto $C_{\tau_B\otimes h}(M)=B\otimes C_h(L^\infty(\mathbb{G}))$ $(:=\tilde{M})$. We denote by $\tilde{A}$ the image of $C_{\tau_A}(A)$ in $\tilde{M}$. 
Put $\tilde{P}:={\cal N}_{\tilde{M}}(\tilde{A})''$. 
Note that there is a faithful normal conditional expectation $E_{\tilde{P}}$ from $\tilde{M}$ onto $\tilde{P}$, since $\tilde{A}$ is an image of an expectation (use $\cite[\rm{Theorem\ IX.4.2}]{Ta2}$).
\begin{Lem}
Under the setting above, for any projection $z\in{\cal Z}(A)\cap (1_B\otimes L^\infty(\mathbb{G}))$ (possibly $z=1$) we have either 
\begin{itemize}
	\item[$\rm (i)$] $Az\preceq_M B$;
	\item[$\rm (ii)$] there exists a conditional expectation from $z(B\otimes \mathbb{B}(L^2(C_h)))z$ onto $z\tilde{P}z$, where we regard $z\in 1_B\otimes C_h$ by the canonical inclusion $L^\infty(\mathbb{G})\subset C_h$.
\end{itemize}
\end{Lem}
\begin{proof}
Suppose $Az\not\preceq_M B$. Then by the same manner as in $\cite[\textrm{Proposition 2.10}]{BHR12}$, we have $\tilde{A}zq\not\preceq_{\tilde{M}}B\otimes L\mathbb{R}r$ for any projections $q\in {\cal Z}(\tilde{A})$ and $r\in L\mathbb{R}$ with $(\tau_B\otimes\mathrm{Tr})(q)<\infty$ and $\mathrm{Tr}(r)<\infty$. 
By the comment below Theorem \ref{popa embed2}, this means $\tilde{A}zq\not\preceq_{\tilde{M}}B\otimes L\mathbb{R}$ for any projection $q\in {\cal Z}(\tilde{A})$ with $(\tau_B\otimes\mathrm{Tr})(q)<\infty$.
We fix such $q$ and assume $z=1$ for simplicity.

Apply Theorem \ref{B} to them and get that $L^2(q\tilde{M})$ is left ${\cal N}_{q\tilde{M}q}(q\tilde{A}q)''$-amenable as a $q\tilde{M}q$-$B$-bimodule. 
Note that ${\cal N}_{q\tilde{M}q}(q\tilde{A}q)''=q\tilde{P}q$ (e.g.\ $\cite[\rm Lemma\ 2.2]{FSW10}$). 
By $\cite[\textrm{Proposition 2.4}]{PV12}$ this means that ${}_{q\tilde{M}q}L^2(q\tilde{M}q)_{q\tilde{P}q} \prec {}_{q\tilde{M}q}L^2(q\tilde{M})\otimes_B L^2(\tilde{M}q)_{q\tilde{P}q}$. 
Let $\nu_q$ (or $\nu$ for $q=1$) be the following canonical multiplication $*$-homomorphism:
\begin{alignat*}{5}
& \hspace{5em}\mathbb{B}(\tilde{q}K)     && \hspace{1em}\mathbb{B}(qq^{\rm op}L^2(\tilde{M}))\\
& \hspace{6em}\cup                             && \hspace{3.5em}\cup                                              \\
& \textrm{$*$-alg}\{q\tilde{M}q\otimes_B q^{\rm op}, q\otimes_B (q\tilde{P}q)^{\rm op}\} &\quad\xrightarrow{\nu_q}\quad& \textrm{$*$-alg}\{q\tilde{M}q, (q\tilde{P}q)^{\rm op}\}&
\end{alignat*}
where $K:=L^2(\tilde{M})\otimes_B L^2(\tilde{M})$ and $\tilde{q}:=(q\otimes_B1)(1\otimes_B q^{\rm op})\in\mathbb{B}(K)$. The weak containment above means that $\nu_q$ is bounded.
Let $(q_i)_i$ be a net of $(\tau_B\otimes\mathrm{Tr})$-finite projections in ${\cal Z}(\tilde{A})$ which converges to 1 strongly. Then since each $q_i$ satisfies the weak containment above, we have for any $x\in \textrm{$*$-alg}\{\tilde{M}\otimes_B1, 1\otimes_B (\tilde{P})^{\rm op}\}\subset\mathbb{B}(K)$,
\begin{equation*}
\|x\|_{\mathbb{B}(K)}=\sup_i\|\tilde{q_i}x\tilde{q_i}\|_{\mathbb{B}(\tilde{q_i}K)}\geq
\sup_i\|\nu_{q}(\tilde{q_i}x\tilde{q_i})\|_{\mathbb{B}(q_iq_i^{\rm op}L^2(\tilde{M}))}=\|\nu(x)\|_{\mathbb{B}(L^2(\tilde{M}))}.
\end{equation*}
Hence $\nu$ is bounded and we have 
${}_{\tilde{M}}L^2(\tilde{M})_{\tilde{P}} \prec {}_{\tilde{M}}L^2(\tilde{M})\otimes_B L^2(\tilde{M})_{\tilde{P}}$. 
(For a general $z$, we have ${}_{z\tilde{M}z}L^2(z\tilde{M}z)_{z\tilde{P}z} \prec {}_{z\tilde{M}z}L^2(z\tilde{M})\otimes_B L^2(\tilde{M}z)_{z\tilde{P}z}$.)

By Arveson's extension theorem, we extend $\nu$ on $C^*\{\tilde{M}\otimes_B1, 1\otimes_B (B'\cap \mathbb{B}(L^2(\tilde{M})))\}$ as a u.c.p.\ map into $\mathbb{B}(L^2(\tilde{M}))$ and denote it by $\Phi$. 
Then since $\tilde{M}\otimes 1$ is contained in the multiplicative domain of $\Phi$, the image of $1\otimes_B (B'\cap \mathbb{B}(L^2(\tilde{M})))$ is contained in $\tilde{M}'=\tilde{M}^{\rm op}$.
Consider the following composition map:
\begin{equation*}
B^{\rm op}\otimes \mathbb{B}(L^2(C_h))= (B'\cap\mathbb{B}(L^2(\tilde{M})))\simeq1\otimes_B (B'\cap\mathbb{B}(L^2(\tilde{M})))\xrightarrow{\Phi}\tilde{M}^{\rm op}\xrightarrow{E_{\tilde{P}}^{\rm op}} \tilde{P}^{\rm op}. 
\end{equation*}
(For a general $z$, we have $z^{\rm op}(B^{\rm op}\otimes \mathbb{B}(L^2(C_h)))z^{\rm op}\rightarrow z(z\tilde{M}z)^{\rm op}\simeq (z\tilde{M}z)^{\rm op}\rightarrow (z\tilde{P}z)^{\rm op}.$)
Finally composing with $\mathrm{Ad}(J_{\tilde{M}})=\mathrm{Ad}(J_B\otimes J_{C_h})$, we get a conditional expectation from $B\otimes \mathbb{B}(L^2(C_h))$ onto $\tilde{P}$.
\end{proof}

Now assume that $L^\infty(\mathbb{G})$ is non amenable and $A\subset M$ is a Cartan subalgebra. 
Let $w$ be a central projection in $L^\infty(\mathbb{G})$ such that  $L^\infty(\mathbb{G})w$ has no amenable direct summand. Write $z=1_B\otimes w$. Then $z\in A$ since $A$ is maximal abelian. 
Since $\tilde{A}$ is Cartan in $\tilde{M}$ (e.g.\ $\cite[\textrm{Subsection 2.3}]{HR10}$), we have $\tilde{M}=\tilde{P}$. 
The lemma above says that we have either (i) $Az\preceq_MB$ or (ii) there exists a conditional expectation from $B\otimes \mathbb{B}(wL^2(C_h))$ onto $z\tilde{M}z=B\otimes C_hw$. 
If (ii) holds, then composing with $\tau_B\otimes \mathrm{id}_{C_h}$, we get a conditional expectation from $\mathbb{B}(wL^2(C_h))$ onto $C_hw$. This is a contradiction since $C_hw=(L^\infty(\mathbb{G})w)\rtimes \mathbb{R}$ is non amenable. 
If (i) holds, then we have $Az\preceq_{Mz}B\otimes \mathbb{C}w$ since $z\in {\cal Z}(M)$, and hence $(B\otimes \mathbb{C}w)'\cap Mz \preceq_{Mz} A'\cap Mz$ by $\cite[\textrm{Lemma 3.5}]{Va08}$ (this is true for non finite $M$). This means $1_B \otimes L^\infty(\mathbb{G})w\preceq_{Mz} Az$ and hence $L^\infty(\mathbb{G})w$ has an amenable direct summand. Thus we get a contradiction.

Next assume that $B=\mathbb{C}$. Let $N\subset L^\infty(\mathbb{G})(=M)$ be a non-amenable subalgebra with expectation $E_N$ and assume that $A\subset N$ is a Cartan subalgebra with expectation $E_A$. 
Let $z$ be a central projection in $N$ such that  $Nz$ is non-amenable and diffuse. Since $A$ is maximal abelian, $z\in A$ and $Az$ is a Cartan subalgebra in $Nz$.  
Hence $Az$ is diffuse, which means $Az\not\preceq_{M}\mathbb{C}$. The above lemma implies that there exists a conditional expectation from $\mathbb{B}(zL^2(C_h))$ onto $z{\cal N}_{\tilde{M}}(\tilde{A})''z={\cal N}_{z\tilde{M}z}(\tilde{A}z)''$. 
Now ${\cal N}_{z\tilde{M}z}(\tilde{A}z)''$ is non amenable and hence a contradiction. In fact, we have inclusions $C_{\tau_A}(A)z\subset C_{\tau_A\circ E_A}(N)z\subset zC_{\tau_A\circ E_A \circ E_N}(M)z$, and since the first inclusion is Cartan, the normalizer of $\tilde{A}z$ in $z\tilde{M}z$ generates a non-amenable subalgebra.

\section{\bf Further remarks}

\subsection{\bf Continuous and discrete cores}
In $\cite[\textrm{Subsection 5.2}]{Is12_2}$, we discussed semisolidity and non solidity of continuous cores of free quantum group $\rm III_1$ factors. Non solidity of them follow from non amenability of the relative commutant of the diffuse algebra $L\mathbb{R}$, which contains a non-amenable centralizer algebra. 

Let $\mathbb{G}$ be a compact quantum group as in Theorem \ref{C} and assume that $L^\infty(\mathbb{G})$ is a full $\rm III_1$ factor and the Haar state is $\mathrm{Sd}(L^\infty(\mathbb{G}))$-almost periodic. Denote its discrete core by $D(L^\infty(\mathbb{G}))$ (see $\cite{Co74}\cite{Dy94}$). 
Then $qD(L^\infty(\mathbb{G}))q$ is always strongly solid for any trace finite projection $q$, 
since $D(L^\infty(\mathbb{G}))$ is isomorphic to $L^\infty(\mathbb{G})_h\otimes \mathbb{B}(H)$ for some separable Hilbert space $H$ and $L^\infty(\mathbb{G})_h$ is a strongly solid $\rm II_1$ factor (this follows from Theorem \ref{B}). Hence structures of these cores are different. In the subsection, we observe this difference from some viewpoints.

Write as $\Gamma\leq \mathbb{R}^*$ the Sd-invariant of $L^\infty(\mathbb{G})$. 
Then regarding $\mathbb{R}\subset \hat{\Gamma}$, take the crossed product $L^\infty(\mathbb{G})\rtimes \hat{\Gamma}$ by the extended modular action of the Haar state. 
This algebra is an explicit realization of the discrete core, namely, $D(L^\infty(\mathbb{G}))\simeq L^\infty(\mathbb{G})\rtimes \hat{\Gamma}$.
Since this extended action is the modular action on the dense subgroup $\mathbb{R}$, we can easily verify that $L^\infty(\mathbb{G})\rtimes \hat{\Gamma} (=:M)$ satisfies the same condition as condition $\rm (AOC)^+$  replacing $\mathbb{R}$-action with $\hat{\Gamma}$-action.
Under this setting, we can prove Theorem \ref{B}, namely, for any amenable subalgebra $A\subset qMq$ (with trace finite projection $q\in L\hat{\Gamma}$), we have either (i) $A\preceq_{M} L\hat{\Gamma}$ or (ii) ${\cal N}_{qMq}(A)''$ is amenable. 
This obviously implies strong solidity of $qMq$, since all diffuse subalgebra $A$ automatically satisfies $A\not\preceq_{qMq}L\hat{\Gamma}$ (because $L\hat{\Gamma}$ is atomic). 

Roughly speaking, this observation implies that the only obstruction of the solidity of $C_h(L^\infty(\mathbb{G}))(=:C_h)$ is a diffuse subalgebra $L\mathbb{R}$. 
More precisely, $C_h$ is non solid since the subalgebra $L\mathbb{R}$ is diffuse, and $D(L^\infty(\mathbb{G}))$ is strongly solid since the subalgebra $L\hat{\Gamma}$ is atomic. 

The next view point is from property Gamma (or non fullness). It is known that a  $\rm II_1$ factor $M$ has the property Gamma if and only if $C^*\{M,M' \}\cap \mathbb{K}(L^2(M))=0$ $\cite{aaa}$. 
Moreover it is not difficult to see that a von Neumann algebra $M$ satisfies condition (AO) and $C^*\{M,M' \}\cap \mathbb{K}(L^2(M))=0$, then it is amenable. 
Hence any non amenable $\rm II_1$ factor can not have property Gamma and condition (AO) at the same time. 
When $L^\infty(\mathbb{G})$ is a full $\rm III_1$ factor, on the one hand, for any $\mathrm{Tr}$-finite projection $p$, $pC_hp$ is non full, since it admits a non trivial central sequence (use the almost periodicity of $h$). 
Hence we can deduce $C^*\{C_h,C_h' \}\cap \mathbb{K}(L^2(C_h)=0$. In particular $C_h$ (and $pC_hp$ for \textit{any} projection $p$) does not satisfy condition (AO). 
On the other hand, for the discrete core $D(L^\infty(\mathbb{G}))$, $qD(L^\infty(\mathbb{G}))q$ always satisfies condition $\rm (AO)^+$ for any projection $q\in L\hat{\Gamma}\simeq \ell^\infty(\Gamma)$ with finite support. 
This is because $D(L^\infty(\mathbb{G}))$ satisfies condition $\rm (AO)^+$ with respect to the quotient $\mathbb{K}(L^2(\mathbb{G}))\otimes \mathbb{B}(\ell^2(\Gamma))$ and hence $qD(L^\infty(\mathbb{G}))q$ satisfies $\rm (AO)^+$ for $\mathbb{K}(L^2(\mathbb{G}))\otimes q\mathbb{B}(\ell^2(\Gamma))q=\mathbb{K}(L^2(\mathbb{G}))\otimes \mathbb{K}(q\ell^2(\Gamma))$. 
By $\cite[\textrm{Theorem B}]{Is12_2}$, we again get the strong solidity of $qD(L^\infty(\mathbb{G}))q$.
We also get fullness of $qD(L^\infty(\mathbb{G}))q$ (since this is non amenable) for any $q\in \ell^\infty(\Gamma)$ with finite support and hence that of $D(L^\infty(\mathbb{G}))$.

\subsection{\bf Primeness of crossed products by bi-exact quantum groups}
In this subsection, we consider only compact quantum groups of Kac type. Let $M=\hat{\mathbb{G}}\ltimes B$ be a finite von Neumann algebra as in the statement of Theorem \ref{A}. Assume that $B$ is amenable. 
Then we have for any amenable subalgebra $A\subset qMq$ $(q\in M$ is a projection), we have either (i) $A\preceq_M B$ or (ii) ${\cal N}_{qMq(A)''}$ is amenable. 
This is a sufficient condition to semisolidity when $B$ is abelian, and to semiprimeness, which means that any tensor decomposition has an amenable tensor component, when $B$ is non abelian. 
Thus we proved that any non amenable von Neumann subalgebra $N\subset p(\hat{\mathbb{G}}\ltimes B)p$ is prime when $B$ is abelian, and semiprime when $B$ is amenable.





\end{document}